\documentclass{amsart}

\usepackage{amsfonts}
\usepackage{amsmath}
\usepackage{amssymb}
\usepackage{enumerate}
\usepackage[margin=3cm]{geometry}
\usepackage[all,cmtip]{xy}
\usepackage[english]{babel}
\usepackage{tikz}
\usepackage{tikz-cd}
\usepackage{tcolorbox}
\usepackage{braket}
\usepackage{subcaption}

\newtheorem{theorem}{Theorem}[section]
\newtheorem{lemma}[theorem]{Lemma}
\newtheorem{corollary}[theorem]{Corollary}
\newtheorem{prop}[theorem]{Proposition}

\theoremstyle{definition}

\newtheorem{definition}[theorem]{Definition}
\newtheorem{example}[theorem]{Example}

\newcommand{\diam}{\text{diam }}
\newcommand{\ep}{\varepsilon}
\newcommand{\bR}{\mathbb R}

\newcommand{\bC}{\mathbb C}

\newcommand{\bZ}{\mathbb Z}
\newcommand{\bN}{\mathbb N}
\newcommand{\bT}{\mathbb T}

\newcommand{\B}{\mathcal B}

\newcommand{\F}{\mathcal F}

\renewcommand{\P}{\mathcal P}

\newcommand{\U}{\mathcal U}

\newcommand{\FOR}{\text{ for }}
\newcommand{\AND}{\text{ and }}
\newcommand{\FORAL}{\text{ for all }}

\newcommand{\OR}{\text{ or }}

\renewcommand{\to}{\rightarrow}

\newcommand{\mt}{\varnothing}
\newcommand{\ol}{\overline}

\renewcommand{\phi}{\varphi}
\newcommand{\upchi}{{\raise.35ex\hbox{$\chi$}}}

\newcommand{\id}{\text{id}}

\makeatletter
\@namedef{subjclassname@2020}{%
  \textup{2020} Mathematics Subject Classification}
\makeatother

\begin{document}
%\issueinfo{vv}{n}{yyyy} % vv is the volume, n is the number, yyyy is the year
%\commby{Stefaan Vaes}% Editor is the name of the Editor who accepted the article
%\pagespan{bbb}{eee}% bbb is the beginning page, eee is the ending page
%\date{Month dd, yyyy}% This is date of receipt of the article
%\revision{Month dd, yyyy}% This is date(s) of revision of the manuscript
\title[Some examples of factor groupoids]{Some examples of factor groupoids}% The title of the article
%\dedicatory{Dedicated to ...}% Only if needed
\author[Mitch Haslehurst]{Mitch Haslehurst}% For multipe authors please use {\protec \and} sequence
%\address{Mitch Haslehurst}
\email{mjthasle@gmail.com}

\begin{abstract} We examine two classes of examples of Hausdorff \'etale factor groupoids; one comes from taking a quotient space of the unit space of an AF-groupoid, and the other comes from certain nonhomogeneous extensions of Cantor minimal systems considered by Robin Deeley, Ian Putnam, and Karen Strung. The reduced C$^*$-algebras of the factor groupoids are classifiable in the Elliott scheme, and we describe their $K$-theory and traces.
\end{abstract}

%\begin{subjclass}
%Primary 46L55, Secondary 46L35, 46L80, 37B05.
%\end{subjclass}

%\begin{keywords}
%C$^*$-algebra, groupoid, K-theory, dynamical system, iterated function system.
%\end{keywords}


\maketitle

\section{Introduction}

In recent years, the problem of constructing groupoid models for C$^*$-algebras has attracted a great deal of attention, not least because groupoid models make tools from dynamics and symmetry available that can help to elucidate the structure of C$^*$-algebras. If we restrict our attention to C$^*$-algebras that are classifiable in the Elliott scheme (that is, all unital, separable, simple C$^*$-algebras with finite nuclear dimension and that satisfy the Universal Coefficient Theorem), the problem becomes finding groupoid models while aiming for a specified Elliott invariant (in other words, ``prescribed" $K$-theory). The most notable of recent results within this line of research is Li's construction in \cite{li}, where a twisted groupoid is produced whose C$^*$-algebra yields a given Elliott invariant. The idea is essentially to modify a similar construction due to Elliott using inductive limits, while modifying the process so as to produce a Cartan subalgebra in the limit, hence a groupoid model. The results are quite extensive, and the complexity of the prescribed invariant is tied intimately with the dimension of the unit space of the constructed groupoid, as well as the necessity of a twist. In \cite{pres}, Putnam constructed groupoids whose C$^*$-algebras have $K$-theory groups that can be realized as dimension groups using a procedure with subgroupoids. The idea is to enlarge an AF equivalence relation to regard it as an open subgroupoid of a larger groupoid, which produces two C$^*$-algebras, one contained in the other as a subalgebra. Other subgroupoids constructed for similar purposes may be found in \cite{equiv}. The methods in \cite{equiv} improve upon those in \cite{li} and \cite{pres} in that it is possible to prescribe $K$-theory with torsion without a twist.

In \cite{equiv} and \cite{pres}, the crucial tool in computing the $K$-theory is Putnam's excision theorem in \cite{exc}. While stated in a great deal of generality, the main result of \cite{exc} has two intended uses specific to groupoids. The first is the so-called ``subgroupoid" situation mentioned above which is used in \cite{equiv} and \cite{pres}. The second is the so-called ``factor groupoid" situation, that is, two groupoids $G$ and $G'$ with a surjective groupoid homomorphism $\pi:G\to G'$ which results in an inclusion $C_r^*(G')\subseteq C_r^*(G)$. The term comes from the theory of dynamical systems via factor and extension systems, and therefore factor groupoids may be considered somewhat intuitively as the noncommutative version of the categorical equivalence between quotient maps $\pi:X\to Y$ of compact Hausdorff spaces and inclusions $C(Y)\subseteq C(X)$, see Proposition \ref{factor}.

The factor groupoid situation has recently begun to see some exploration. In \cite{matui}, Matui has established an analogue of Putnam's excision theorem for groupoid homology when the groupoids have totally disconnected unit spaces. These results are then used to compute the homology groups arising from examples of shifts of finite type and hyperplane groupoids. In the present paper, our goal is to produce examples of factor groupoids whose C$^*$-algebras are classifiable, and examine their $K$-theory. The unit spaces will often not be totally disconnected: connected components will be circles, or subsets of attractors of iterated function systems.

We will consider two distinct constructions of factor groupoids. The first construction begins with the same set-up as in \cite{pres}: two Bratteli diagrams $(V,E)$ and $(W,F)$ and two graph embeddings of $(W,F)$ into $(V,E)$ satisfying some specified conditions. The goal is to create a groupoid whose ordered $K_0$-group is the dimension group of $(V,E)$, and whose $K_1$-group is the dimension group of $(W,F)$. While the construction in \cite{pres} enlarges the usual tail equivalence relation $R_E$ on the path space $X_E$ of $(V,E)$ to include pairs of paths in the two embeddings (which creates a new groupoid over the same unit space), we actually collapse these paths in $X_E$ and pass $R_E$ through the resulting quotient map, creating an equivalence relation on an entirely new space, and we denote the space and equivalence relation by $X_\xi$ and $R_\xi$, respectively. It should also be noted that, unlike in \cite{pres}, obtaining the \'etale topology and amenability on the new groupoids is straightforward (these properties essentially carry over right from $R_E$). It is the new unit spaces, rather than the new groupoids, that require more attention in order to describe them. We address this in section 3.

\begin{theorem}\label{embedding}
Let $(V,E)$ and $(W,F)$ be two Bratteli diagrams satisfying the embedding conditions described in section 3. There exists a quotient space $X_\xi$ of $X_E$ that satisfies the following.
\begin{enumerate}[(i)]
    \item $X_\xi$ is compact and metrizable.
    \item Every connected component of $X_\xi$ is either a single point or homeomorphic to $\bT$.
    \item $X_\xi$ has covering dimension one.
\end{enumerate}
Moreover, the resulting factor groupoid $R_\xi$ of $R_E$ satisfies the following.
\begin{enumerate}[(i)]
    \item $R_\xi$ is  second-countable, locally compact, Hausdorff, \'etale, and principal with dynamic asymptotic dimension zero.
    \item If $R_E$ is minimal, then so is $R_\xi$, and hence $C_r^*(R_\xi)$ is classifiable.
    \item $K_0(C_r^*(R_\xi))$ is order isomorphic to $K_0(C_r^*(R_E))$.
    \item $K_1(C_r^*(R_\xi))$ is isomorphic to $K_0(C_r^*(R_F))$.
    \item The map $\tau\mapsto\tau\circ\alpha$, where $\tau$ is a tracial state on $C_r^*(R_E)$ and $\alpha$ is as in Proposition \ref{factor}, is an affine homeomorphism from $T(C_r^*(R_E))$ to $T(C_r^*(R_\xi))$.
\end{enumerate}
\end{theorem}

While the order isomorphism between the $K_0$-groups in Theorem \ref{embedding} is induced by the inclusion $C_r^*(R_\xi)\subseteq C_r^*(R_E)$, the $K_1$-group appears to receive its structure from the unit space of $R_\xi$, whose connected components are either points or circles. Indeed, the isomorphism between $K_1(C_r^*(R_\xi))$ and $K_0(C_r^*(R_F))$ is essentially given by a Bott map, and the elements of the group $K_0(C_r^*(R_F))$ seem to correspond to winding numbers on (at least a portion of) the space $X_\xi$. Curiously, these C$^*$-algebras are isomorphic (by classification considerations) to those constructed in \cite{pres} using subgroupoid methods, but the two methods produce non-isomorphic Cartan subalgebras.

From Theorem \ref{embedding} we obtain the following result of prescribed $K$-theory.

\begin{corollary}\label{pres}
Let $G_0$ be a simple, acyclic, dimension group and $G_1$ a countable, torsion-free abelian group. Then there exist Bratteli diagrams $(V,E)$ and $(W,F)$ satisfying the embedding conditions described in section 3 such that the resulting $R_\xi$ from Theorem \ref{embedding} has $K_0(C_r^*(R_\xi))\cong G_0$ as ordered groups with order unit and $K_1(C_r^*(R_\xi))\cong G_1$.
\end{corollary}

The proof is essentially to find Bratteli diagrams $(V,E)$ and $(W,F)$ whose dimension groups are $G_0$ and $G_1$, respectively. By telescoping and symbol splitting the diagram $(V,E)$, one can arrange that the embeddings described in section 3 exist without altering the dimension groups. Details may be found in \cite{pres}.

Admittedly, the collection of Elliott invariants that are obtainable via Theorem \ref{embedding} and Corollary \ref{pres} is not very extensive, as we are restricted to $K$-groups that are dimension groups and we have little control over the pairing between $K$-theory and traces. However, it should be observed that the factor groupoids $R_\xi$ have a relatively simple structure; indeed, the entire construction is essentially a generalized ``noncommutative" version of the Cantor function (see the map $\phi$ in Proposition \ref{prod}). This raises the question as to what $K$-theory data may be obtained with more complicated factor maps, or with additional structure(s) such as twists.

We also obtain a result on AF-embeddability. We say that a C$^*$-algebra is AF-embeddable if it can be embedded into an AF-algebra. There has been a considerable amount of research conducted into determining when a C$^*$-algebra is AF-embeddable, initiated by the result of Pimsner and Voiculescu that every irrational rotation algebra is AF-embeddable \cite{afembed}. Possibly the most substantial recent result is the following, due to Schafhauser \cite{schafhauser}: every separable, exact C$^*$-algebra satisfying the Universal Coefficient Theorem (UCT) and possessing a faithful, amenable trace is AF-embeddable. Though the following result is weaker, it seems worth stating because it follows immediately from Theorem \ref{embedding} and Corollary \ref{pres}, and because the embedding is given concretely by a factor map.

\begin{corollary}\label{afembed}
Suppose that $A$ is a classifiable C$^*$-algebra with $(K_0(A),K_0(A)^+,[1_A])$ a simple, acyclic dimension group, $K_1(A)$ a countable, torsion-free abelian group, and standard pairing between $K$-theory and traces. Then $A$ is AF-embeddable.
\end{corollary}

The second construction we consider arises from the nonhomogeneous extensions of Cantor minimal systems considered in \cite{dps}, which we refer to as DPS extensions in acknowledgement of the authors. In these extensions, the fibres of the factor map are either a single point or homeomorphic to the attractor of a given iterated function system. Armed with the excision theorem from \cite{exc}, we are in a position to provide a description of the $K$-theory of these extensions. A minor modification of the extension is necessary in order for the factor map to satisfy the regularity hypothesis needed to apply the excision theorem, and we describe the construction of the extension and its modification in section 4.

If $(X,\phi)$ is a dynamical system, we denote the orbit equivalence relation by $R_\phi$.

\begin{theorem}\label{ifs}
Let $(C,d_C,\F)$ be a compact invertible iterated function system and $(X,\phi)$ a Cantor minimal system. There is a modified version of the associated DPS extension $(\tilde X,\tilde\phi)$ of $(X,\phi)$, an AF-algebra $A$, a short exact sequence
\begin{center}
    % https://tikzcd.yichuanshen.de/?fbclid=IwAR0FybHEYl6QS_hMmNoKyCn60muRw7aO-eQAaFVrU3mcQkgAWpX_FaPCKwY#N4Igdg9gJgpgziAXAbVABwnAlgFyxMJZAFgBoAGAXVJADcBDAGwFcYkQBpAfXIAoBhLgCcAegCpeAHUl5GsAAQANaXgC28aQCMAWgEpdIAL6l0mXPkIoAbBWp0mrdtz6DREgEpcAoiP77pEGrw8s4CAvpGJiAY2HgERAActjQMLGyIIOSRprEWROTJ9mnsWcY55vEoAEyFqY4Zoa7ivJ5eEYZ2MFAA5vBEoABmQhCqSADMNDgQSMRlIEMjM5PTiORzC6Ory0gAjOvDmzvbiFUdhkA
\begin{tikzcd}
0 \arrow[r] & K_0(C_r^*(R_\phi)) \arrow[r] & K_0(C_r^*(R_{\tilde\phi})) \arrow[r] & K_0(A)\otimes(K^0(C)/\bZ) \arrow[r] & 0
\end{tikzcd}
\end{center}
and a short exact sequence
\begin{center}
    % https://tikzcd.yichuanshen.de/?fbclid=IwAR0FybHEYl6QS_hMmNoKyCn60muRw7aO-eQAaFVrU3mcQkgAWpX_FaPCKwY#N4Igdg9gJgpgziAXAbVABwnAlgFyxMJZAFgBoAGAXVJADcBDAGwFcYkQBpAfQEYAKAMJcATgD0AVHwA6UvI1gACABoy8AW3gyARgC0AlHpABfUuky58hFADYK1Ok1btu5QSIl8ASlwCiogQYyEOrwCtz8AoIGxqYgGNh4BEQAHHY0DCxsiCDkMWYJlkTkaQ6Z7Lkm+RZJKABMJRlO2do6xvYwUADm8ESgAGbCEGpIAMw0OBBIxJUgA0NT45OI5DNzw8uLSDyrg+s8m4i1RpRGQA
\begin{tikzcd}
0 \arrow[r] & \bZ \arrow[r] & K_1(C_r^*(R_{\tilde\phi})) \arrow[r] & K_0(A)\otimes K^{-1}(C) \arrow[r] & 0.
\end{tikzcd}
\end{center}
\end{theorem}

By $K^0(C)/\bZ$ we mean the quotient of $K^0(C)$ by the canonical copy of $\bZ$ generated by the unit of $C(C)$. We refine the conclusion of Theorem \ref{ifs} in some particular cases of interest, such as when $C$ is a closed cube in Euclidean space (this case was the inital motivation for the construction in \cite{dps}).

\begin{corollary}\label{ifs2}
In the context of Theorem \ref{ifs}, we have the following.
\begin{enumerate}[(i)]
    \item If $C$ is contractible, then the injective maps in both sequences are isomorphisms.
    \item If $X_E$ consists of only one path $(x_1,x_2,\ldots)$ of edges such that\emph{$f_{x_n}=\id_C$} for all $n\geq1$, then the AF-algebra $A$ satisfies $K_0(A)\cong\bZ$ and hence
    \[K_0(C_r^*(R_{\tilde\phi}))/K_0(C_r^*(R_\phi))\cong K^0(C)/\bZ\qquad K_1(C_r^*(R_{\tilde\phi}))/\bZ\cong K^{-1}(C)\]
\end{enumerate}
\end{corollary}

%Theorem \ref{ifs} provides a means of producing integer actions on compact metric spaces that have some given $K$-theory data using attractors of iterated function systems. Of course, this requires knowledge about $K^*(C)$, where $C$ is such an attractor. Most classical examples of single-matrix affine iterated function systems seem to produce an attractor that is either contractible, totally disconnected, or has $K$-groups that are infinitely generated and free abelian, as illustrated in the examples.

At this point it is not clear (at least to the author) whether or not the conclusions of Theorem \ref{ifs} and Corollary \ref{ifs2} can be improved upon without knowing more about $C$. For example, it is not clear precisely when the short exact sequences split, and hence when the quotients in Corollary \ref{ifs2} can be replaced with direct sums. Of course, this is trivially the case in part (i) of the corollary, and also when both $K^0(C)$ and $K^{-1}(C)$ are free abelian, such as when $C$ is the Sierpi\'nski triangle, see Example \ref{sier}.

The paper is organized as follows. In section 2 we outline the necessary preliminary material. In section 3 we construct and analyze the spaces $X_\xi$ and groupoids $R_\xi$. In section 4 we describe the DPS extensions and their modifications. In section 5 we analyze the $K$-theory and complete the proofs of Theorem \ref{embedding} and Theorem \ref{ifs}.

\section{Preliminaries}

We refer the reader to \cite{groupoid} for a detailed treatment of Hausdorff \'etale groupoids, but we outline our notation here. Let $G$ be a locally compact Hausdorff groupoid with composable pairs $G^{(2)}$, unit space $G^{(0)}$, and range and source maps $r,s:G\to G^{(0)}$ defined by $r(x)=xx^{-1}$ and $s(x)=x^{-1}x$. We will only be concerned with \'etale groupoids, which means that $r$ and $s$ are local homeomorphisms. An open subset $U$ of $G$ upon which $r$ and $s$ are local homeomorphisms is called an open bisection. For $u$ in $G^{(0)}$, we denote $r^{-1}(u)$ and $s^{-1}(u)$ by $G^u$ and $G_u$, respectively, which are discrete in the relative topology from $G$.

A groupoid homomorphism $\pi:G\to H$ is a map such that $\pi\times\pi(G^{(2)})\subseteq H^{(2)}$ and $\pi(xy)=\pi(x)\pi(y)$ for every pair $(x,y)$ in $G^{(2)}$. If $G$ and $H$ are topological groupoids, we will require that $\pi$ be continuous. We have $\pi(x^{-1})=\pi(x)^{-1}$ for all $x$ in $G$, and $\pi(G^{(0)})\subseteq H^{(0)}$, with equality if $\pi$ is surjective. We also have $\pi\circ r=r\circ\pi$ and $\pi\circ s=s\circ\pi$. If $\pi$ is bijective and its inverse is continuous, we say it is a groupoid isomorphism.

We denote the $^*$-algebra of all continuous compactly supported complex-valued functions on $G$ by $C_c(G)$ with the operations
\[(f\star g)(x)=\sum_{y\in G^{r(x)}}f(y)g(y^{-1}x)\qquad f^*(x)=\ol{f(x^{-1})}\]
Since $G^{r(x)}$ is discrete and $f$ and $g$ are compactly supported, the above sum is finite for every $x$ in $G$. For $u$ in $G^{(0)}$ and $x$ in $G_u$, let $\delta_x$ be the element of $l^2(G_u)$ that is $1$ at $x$ and $0$ elsewhere. Define the representation $\psi_u:C_c(G)\to\B(l^2(G_u))$ by
\[\psi_u(f)\delta_x=\sum_{y\in G_{r(x)}}f(y)\delta_{yx}\]
for $f$ in $C_c(G)$ and $x$ in $G_u$. The representation $\bigoplus_{u\in G^{(0)}}\psi_u$ is well-defined, it is called the left regular representation of $C_c(G)$, and it is faithful. The closure of $\bigoplus_{u\in G^{(0)}}\psi_u(C_c(G))$ in $\bigoplus_{u\in G^{(0)}}\B(l^2(G_u))$ is called the reduced C$^*$-algebra of $G$, and we denote it by $C_r^*(G)$.

The groupoids that concern us here are equivalence relations. We will either construct them concretely, or obtain them through free actions of the integers. If $R$ is an equivalence relation on a set $X$, then $R$ is a groupoid with product $(x,y)(y',z)=(x,z)$ (which is only defined when $y=y'$) and inverse $(x,y)^{-1}=(y,x)$. If $(X,\phi)$ is a \textit{dynamical system}, that is, $X$ is a set and $\phi:X\to X$ is a bijection, then $X\times\bZ$ is a groupoid with product $(x,l)(y,m)=(x,l+m)$ (which is only defined when $y=\phi^l(x)$) and inverse $(x,l)^{-1}=(\phi^l(x),-l)$. The \textit{orbit equivalence relation} of the dynamical system $(X,\phi)$ is
\[R_\phi=\{(x,\phi^l(x))\mid x\in X\AND l\in\bZ\}\]
and if $\phi$ acts freely on $X$, then the map
\[X\times\bZ\to R_\phi:(x,l)\mapsto(x,\phi^l(x))\]
is, algebraically, an isomorphism of groupoids. If $X$ is a space and $\phi$ is a homeomorphism, then $X\times\bZ$ is \'etale in the product topology, and we endow $R_\phi$ with the unique topology that makes the above algebraic groupoid isomorphism a homeomorphism. An equivalence relation (resp. dynamical system) on $X$ is called \textit{minimal} if every equivalence class (resp. orbit) is dense in $X$. All dynamical systems we consider will be minimal on infinite spaces, hence they will all be free.

The C$^*$-algebra $C_r^*(R_\phi)$ is isomorphic to the (reduced) crossed product $C(X)\times_\phi\bZ$, and thus the Pimsner-Voiculescu sequence \cite{pv} is available to help compute $K_*(C_r^*(R_\phi))$.

It is a rather standard result that, given a continuous, proper, surjective map $\pi:X\to Y$ of locally compact Hausdorff spaces, the map $\alpha:C_0(Y)\to C_0(X)$ defined by $\alpha(f)=f\circ\pi$ is an injective $^*$-homomorphism. We prove a generalization for groupoids, which is our ``factor groupoid" situation.

\begin{prop}\label{factor}
Let $G$ and $G'$ be locally compact, Hausdorff, \'etale groupoids and $\pi:G\to G'$ a continuous, proper, surjective groupoid homomorphism. Suppose also that for all $u$ in $G^{(0)}$, the map $\pi|_{G^u}:G^u\to(G')^{\pi(u)}$ is bijective. Then the map $\alpha:C_c(G')\to C_c(G)$ defined by $\alpha(f)=f\circ\pi$ is an injective $^*$-homomorphism, and it extends to an injective $^*$-homomorphism from $C_r^*(G')$ to $C_r^*(G)$ (which we also denote by $\alpha$). If $\pi$ is a groupoid isomorphism, then $\alpha$ is a $^*$-isomorphism.
\end{prop}

\begin{proof}
That $\alpha(f)$ is continuous and compactly supported follows from the fact that $\pi$ is continuous and proper. For any $x$ in $G$, we have
\[\alpha(f\star g)(x)=(f\star g)(\pi(x))=\sum_{z\in (G')^{\pi(r(x))}}f(z)g(z^{-1}\pi(x))\]
and
\[(\alpha(f)\star\alpha(g))(x)=\sum_{y\in G^{r(x)}}f(\pi(y))g(\pi(y^{-1}x))\]
Each $z$ in the first sum corresponds to one and only one $y$ in the second sum via the bijective map $\pi|_{G^{r(x)}}$. Thus both sums have precisely the same terms and they are equal. It is straightforward to check that $\alpha(f^*)=\alpha(f)^*$ and that it is injective. Now, for each $u$ in $G^{(0)}$ the Hilbert spaces $l^2(G_u)$ and $l^2(G'_{\pi(u)})$ are isomorphic via the unitary $\delta_x\mapsto\delta_{\pi(x)}$ (notice that $\pi|_{G_u}:G_u\to(G')_{\pi(u)}$ is also bijective, for any $u$ in $G^{(0)}$). It follows that the representation $\bigoplus_{u\in G^{(0)}}(\psi_u\circ\alpha)$ is unitarily equivalent (modulo extra direct factors over which the reduced norm is constant) to the left regular representation of $C_c(G')$, so we may identify $C_r^*(G')$ with the closure of $\bigoplus_{u\in G^{(0)}}\psi_u(\alpha(C_c(G')))$, which is contained in $C_r^*(G)$. If $\pi$ is a groupoid isomorphism, then the map $\beta:C_c(G)\to C_c(G')$ defined by $\beta(f)=f\circ\pi^{-1}$ is obviously the inverse of $\alpha$. Thus $C_c(G)$ is (algebraically) $^*$-isomorphic to $C_c(G')$ via $\alpha$, and the closures of $\bigoplus_{u\in G^{(0)}}\psi_u(C_c(G))$ and $\bigoplus_{u\in G^{(0)}}\psi_u(\alpha(C_c(G')))$ are equal.
\end{proof}

If $\pi$ and $\alpha$ are as in Proposition \ref{factor}, it is worth recording the easy but useful observation that
\begin{equation}\label{inj}
    \alpha(C_c(G'))=\{f\in C_c(G)\mid f|_{\pi^{-1}(x')}\text{ is constant for all }x'\in G'\}
\end{equation}

We now discuss relative $K$-theory. If $A$ is a unital C$^*$-algebra and $A'$ is a C$^*$-subalgebra of $A$ that contains the unit of $A$, there is a homology theory $K_*(A',A)$ in the sense that if $A'\subseteq A$ and $B'\subseteq B$, and $\psi:A\to B$ is a $^*$-homomorphism with $\psi(A')\subseteq B'$, then there is an induced homomorphism $\psi_*:K_*(A',A)\to K_*(B',B)$ which is an isomorphism if $\psi$ is a $^*$-isomorphism and $\psi(A')=B'$. The theory satisfies Bott periodicity and there is an exact sequence%consists of classes of triples $[p,q,v]$, where $p$ and $q$ are projections in a matrix algebra over $A'$ and $v$ is a matrix over $A$ with $v^*v=p$ and $vv^*=q$. The group operation is the usual diagonal sum, the identity element of the group is $[0,0,0]$, and the inverse formula is $-[p,q,v]=[q,p,v^*]$. Other useful formulas are $[p,q,v]+[q,r,w]=[p,r,wv]$ and $[p,q,v]=[p,q,w]$ when $v$ and $w$ are homotopic as partial isometries from $p$ to $q$.

%The relative group $K_1(A',A)$ consists of classes of triples $[p,u,g]$, where $p$ is a projection in $M_n(A')$ for some $n$, $u$ is a unitary in $pM_n(A')p$ and $g$ is a unitary in $C[0,1]\otimes pM_n(A)p$ with $g(0)=p$ and $g(1)=u$. Every triple may be represented by a class $(p,u,g)$ where $p=1_n$ for some $n$.

%The main result that will be of use here is the six-term exact sequence of Theorem 2.1 of \cite{haslehurst},

\begin{center}
    % https://tikzcd.yichuanshen.de/?fbclid=IwAR0FybHEYl6QS_hMmNoKyCn60muRw7aO-eQAaFVrU3mcQkgAWpX_FaPCKwY#N4Igdg9gJgpgziAXAbVABwnAlgFyxMJZAFgBoAGAXVJADcBDAGwFcYkQBpAfXIAoBhLgCcAegCpeAJS4AdGQA8sASiUgAvqXSZc+QigBMFanSat23PoNETpcxUoDcV8VK4BRFes0gM2PASJyIxoGFjZETi4ARgFhF2kPVQ0tP11A0n1jULMI7hjnG1kFZSTvXx0AlDJMkNNwyMs4wsSvFIq9ZEMakzDzaNjrVztlJybXFrVjGCgAc3giUAAzIQgAWyRDEBwIJCjkkGW13ZptpHJ9w-XEIK2dxGILlauyW6QAVkejxDeTu4BmT5XP6-DaTNRAA
\begin{tikzcd}
K_1(A) \arrow[rr, "\mu_0"]   &  & K_0(A',A) \arrow[rr, "\nu_0"] &  & K_0(A') \arrow[dd] \\
                             &  &                                         &  &                              \\
K_1(A') \arrow[uu] &  & K_1(A',A) \arrow[ll, "\nu_1"] &  & K_0(A) \arrow[ll, "\mu_1"]  
\end{tikzcd}
\end{center}
where the vertical maps are induced by the inclusion $A'\subseteq A$. The exact sequence above may be obtained by defining $K_*(A',A)$ to be the $K$-theory of the mapping cone of the inclusion, but the portrait described in \cite{haslehurst} via Karoubi's definitions seems to be more prudent if one wishes to use results from \cite{exc}. We now turn to these results.

If $(X,d)$ is a metric space, $Y$ is a subset of $X$, and $\ep>0$, we define $B(Y,\ep)$ to be the set $\{x\in X\mid\text{there exists }y\in Y\text{ such that }d(x,y)<\ep\}$. If $Y$ and $Z$ are closed, bounded subsets of $X$, the \textit{Hausdorff distance} between $Y$ and $Z$ is
\[d(Y,Z)=\inf\{\ep\mid Y\subseteq B(Z,\ep)\AND Z\subseteq B(Y,\ep)\}.\]
The \textit{diameter} of a bounded subset $Y$ of $X$ is $\diam Y=\sup\{d(x,y)\mid x,y\in Y\}$.

\begin{definition}[7.6 of \cite{exc}]\label{Hs}
    Let $\pi:G\to G'$ be as in Proposition \ref{factor}, and suppose that $G$ has a metric $d_G$ yielding its topology.
    \begin{enumerate}[(i)]
        \item Let
        \[H'=\{x'\in G'\mid\#\pi^{-1}(x')\neq1\}\]
        and endow it with the metric
        \[d_{H'}(x',y')=d_G(\pi^{-1}(x'),\pi^{-1}(y')),\]
        where the right side is the Hausdorff distance.
        \item Let $H=\pi^{-1}(H')$, and endow it with the metric
        \[d_H(x,y)=d_G(x,y)+d_{H'}(\pi(x),\pi(y)).\]
        \item For every integer $n\geq1$, let
        \[\textstyle H_n=\{x\in H\mid\diam\pi^{-1}(\pi(x))>\frac1n\},\]
        where the diameter is measured with $d_G$.
    \end{enumerate}
\end{definition}

\begin{definition}[7.7 of \cite{exc}]\label{reg}
$\pi:G\to G'$ is \emph{regular} if, for all $x'$ in $H'$ and all $\ep>0$, there is an open set $U'\subseteq G'$ with $x'$ in $U'$ such that if $y'$ is in $U'$, at least one of the following holds
\[d_G(\pi^{-1}(x'),\pi^{-1}(y'))<\ep\qquad\diam\pi^{-1}(y')<\ep\]
\end{definition}

There is also a necessary notion of ``measure regularity" for $\pi$. We will not record the definition because in our cases it will follow from regularity.

\begin{definition}[5.5 of \cite{exc}]
Let $C\geq1$. We say that a subset $X\subseteq H$ which is closed in $G$ has the $C$-\emph{extension property} if, for any $f$ in $C_c(H)$ with support in $X$, there exists $\tilde f$ in $C_c(G)$ such that $\tilde f|_X=f|_X$ and $\|\tilde f\|_r\leq C\|f\|_r$, where the left norm is the reduced norm on $C_c(G)$, and the right norm is the reduced norm on $C_c(H)$.
\end{definition}

The following will be our main tool for computing the $K$-theory of the factor groupoids.

\begin{theorem}[7.9, 7.18 and 7.19 of \cite{exc}]\label{exc}
Let $\pi:G\to G'$ be as in Proposition \ref{factor} and assume it is regular and measure regular. Then $H'$ and $H$, with the metrics $d_{H'}$ and $d_H$, are locally compact, Hausdorff, \'etale groupoids with finer topologies than the relative topologies received from $G'$ and $G$, respectively, and $\pi|_{H}:H\to H'$ also satisfies the hypotheses of Proposition \ref{factor}. Moreoever, if there exists a $C\geq1$ such that $\ol{ H_n}^G$ (where the closure is taken in $G$) has the $C$-extension property for all $n\geq1$, then
\[K_*(C_r^*(G'),C_r^*(G))\cong K_*(C_r^*(H'),C_r^*(H))\]
\end{theorem}

We outline the necessary definitions and notation for AF-groupoids. Let $(V,E)$ be a Bratteli diagram, that is, an infinite directed graph consisting of a set of vertices $V$, a set of edges $E$, and maps $i,t:E\to V$ such that the following hold.
\begin{enumerate}[(i)]
    \item $V$ and $E$ are ordered and partitioned into countably many finite subsets $V_n$ for $n\geq0$ and $E_n$ for $n\geq1$.
    \item $V_0$ consists of a single vertex $v_0$.
    \item If $e$ is in $E_n$, then $i(e)$ is in $V_{n-1}$ and $t(e)$ is in $V_n$.
    \item $i^{-1}(v)$ is nonempty for each vertex $v$ and $t^{-1}(v)$ is nonempty for each vertex $v$ other than $v_0$.
\end{enumerate}

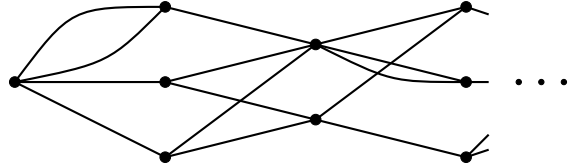
\begin{figure}[ht]
    \begin{tikzpicture}
     \filldraw[black] (0,0) circle (2pt);
     \filldraw[black] (2,1) circle (2pt);
     \filldraw[black] (2,0) circle (2pt);
     \filldraw[black] (2,-1) circle (2pt);
     \filldraw[black] (4,0.5) circle (2pt);
     \filldraw[black] (4,-0.5) circle (2pt);
     \filldraw[black] (6,1) circle (2pt);
     \filldraw[black] (6,0) circle (2pt);
     \filldraw[black] (6,-1) circle (2pt);
     \filldraw[black] (6.7,0) circle (1pt);
     \filldraw[black] (7,0) circle (1pt);
     \filldraw[black] (7.3,0) circle (1pt);
     \draw[black,thick] (0,0) .. controls (0.75,1) .. (2,1);
     \draw[black,thick] (0,0) .. controls (1.25,0.25) .. (2,1);
     \draw[black,thick] (0,0) -- (2,0);
     \draw[black,thick] (2,1) -- (4,0.5);
     \draw[black,thick] (2,-1) -- (4,-0.5);
     \draw[black,thick] (2,-1) -- (4,0.5);
     \draw[black,thick] (0,0) -- (2,-1);
     \draw[black,thick] (2,0) -- (4,0.5);
     \draw[black,thick] (2,0) -- (4,-0.5);
     \draw[black,thick] (4,0.5) -- (6,1);
     \draw[black,thick] (4,0.5) -- (6,0);
     \draw[black,thick] (4,0.5) .. controls (5,0) .. (6,0);
     \draw[black,thick] (4,-0.5) -- (6,1);
     \draw[black,thick] (4,-0.5) -- (6,-1);
     \draw[black,thick] (6,1) -- (6.3,0.9);
     \draw[black,thick] (6,0) -- (6.3,0);
     \draw[black,thick] (6,-1) -- (6.3,-0.7);
     \draw[black,thick] (6,-1) -- (6.3,-0.9);
    \end{tikzpicture}
    \caption{A Bratteli diagram}
    \label{bratteli}
\end{figure}

We say that $(V,E)$ has full edge connections if, for every $n\geq0$ and pair of vertices $v$ in $V_n$ and $w$ in $V_{n+1}$, there is an edge $e$ in $E_{n+1}$ with $i(e)=v$ and $t(e)=w$. We denote by $X_E$ the path space of $(V,E)$: the set of all infinite paths, that is, all sequences $x=(x_1,x_2,x_3,\ldots)$, where $x_n$ is in $E_n$ and $t(x_n)=i(x_{n+1})$ for all $n\geq1$. For a finite sequence of edges $p=(p_1,p_2,\ldots,p_n)$ where $p_j$ is in $E_j$ and $t(p_j)=i(p_{j+1})$, we denote
\[C(p)=C(p_1,p_2,\ldots,p_n)=\{x\in X_E\mid x_j=p_j\FOR1\leq j\leq n\}\]
called the cylinder set of $p$. For any $p$ and $q$, either $C(p)$ and $C(q)$ are disjoint, or one is contained in the other. The path space $X_E$ is a totally disconnected compact metric space in the metric
\[d_E(x,y)=\inf\{1,2^{-n}\mid x_j=y_j\FOR1\leq j\leq n\}\]
and the cylinder sets $C(p)$ are closed and open (clopen) in this metric. If $(V,E)$ has full edge connections, then $X_E$ has no isolated points and is thus homeomorphic to the Cantor set. The equivalence relation $R_E\subseteq X_E\times X_E$ is defined by $(x,y)\in R_E$ if and only if there exists some $n\geq1$ such that $x_k=y_k$ for all $k\geq n$, also known as tail-equivalence. For finite paths $p$ and $q$ with $t(p)=t(q)\in V_n$, let
\[\gamma(p,q)=\{(x,y)\in R_E\mid x\in C(p),y\in C(q),x_k=y_k\FORAL k\geq n+1\}\]
These sets form a base of compact open bisections which makes $R_E$ into a second-countable locally compact Hausdorff \'etale groupoid. If $(V,E)$ has full edge connections, then $R_E$ is minimal. The topology is concretely metrizable: write
\[R_n=\{(x,y)\in R_E\mid x_k=y_k\FORAL k\geq n\}\]
for $n\geq1$ and $R_0=\mt$, and note that $R_E=\bigcup_{n=0}^\infty(R_{n+1}-R_n)$, which is a disjoint union of compact open sets. Let $d_E^{(2)}((x,y),(x',y'))=\max\{d_E(x,x'),d_E(y,y')\}$ if $(x,y)$ and $(x',y')$ are both in $R_{n+1}-R_n$ for some $n$, and $d_E^{(2)}((x,y),(x',y'))=1$ otherwise.

The C$^*$-algebra $C_r^*(R_E)$ is a unital AF-algebra. Concretely, the linear span of the characteristic functions $\upchi_{\gamma(p,q)}$ of the compact open sets $\gamma(p,q)$ act as ``matrix units" and generate a dense union of finite dimensional algebras.

Finally, we need the notion of a Bratteli-Vershik system. An \emph{ordered Bratteli diagram} $(V,E)$ is a Bratteli diagram where the edges are endowed with an order, and two edges $e$ and $f$ are comparable if and only if $t(e)=t(f)$. If there is only one infinite path in $X_E$ that consists of only maximal edges and only one infinite path in $X_E$ that consists of only minimal edges, we say that the diagram $(V,E)$ is \emph{properly ordered}, and in this case, we denote the path of all maximal (resp. minimal) edges in $X_E$ by $x^{\max}$ (resp. $x^{\min}$).

A \emph{Bratteli-Vershik system} is a pair $(X_E,\phi_E)$, where $X_E$ is the path space of a properly ordered Bratteli diagram, and $\phi_E:X_E\to X_E$ is a map that is defined as follows. If $x=(x_1,x_2,x_3,\ldots)$ is in $X_E$ and $n$ is the smallest positive integer such that $x_n$ is not maximal (if such an $n$ exists), then
\[\phi_E(x)=(y_1,y_2,\ldots,y_{n-1},y_n,x_{n+1},x_{n+2},\ldots)\]
where $(y_1,y_2,\ldots,y_{n-1})$ is the unique finite path of minimal edges ending at $t(y_{n-1})$, and $y_n$ is the successor of $x_n$. If no such $n$ exists, then $x=x^{\max}$, and we set $\phi_E(x^{\max})=x^{\min}$.

The following observation will be important later: if $x$ is a path in $X_E$ such that neither $(x,x^{\max})$ nor $(x,x^{\min})$ are in $R_E$, then $\{(x,\phi_E^n(x))\mid n\in\bZ\}\subseteq R_E$. 

\section{Quotients of path spaces}

In this section we describe an equivalence relation, denoted $\sim_\xi$, on the path space $X_E$ of a Bratteli diagram $(V,E)$ and examine the resulting quotient space of $X_E$.

Let $(V,E)$ and $(W,F)$ be two Bratteli diagrams, $(V,E)$ with full edge connections, such that there exist two graph embeddings $\xi^0$ and $\xi^1$ of $(W,F)$ into $(V,E)$. We regard $\xi^0$ and $\xi^1$ as functions on both $W$ and $F$. More precisely, for both $j=0,1$, the following hold.

\begin{enumerate}[(i)]
    \item $\xi^j(w)$ is in $V_n$ for all $w$ in $W_n$ and all $n\geq0$.
    \item $\xi^j(f)$ is in $E_n$ for all $f$ in $F_n$ and all $n\geq1$.
    \item $i(\xi^j(f))=\xi^j(i(f))$ and $t(\xi^j(f))=\xi^j(t(f))$ for all $f$ in $F$.
    \item $\xi^j(f)\neq\xi^j(f')$ for all $f\neq f'$ in $F$.
\end{enumerate}
We also require that

\begin{enumerate}[(i)]
\setcounter{enumi}{4}
    \item $\xi^0(w)=\xi^1(w)$ for all $w$ in $W$.
    \item $\xi^0(F)\cap\xi^1(F)=\mt$.
\end{enumerate}

Due to (v), the functions $\xi^0,\xi^1:W\to V$ are identical, so we may denote them both by $\xi$.

The equivalence relation $\sim_\xi$ on $X_E$ is defined as follows. For $x=(x_1,x_2,x_3,\ldots)$ in $X_E$, suppose that there is a $j$ in $\{0,1\}$, an $n_0\geq0$, and edges $z_n$ in $F_n$ such that $x_n=\xi^j(z_n)$ for $n\geq n_0+1$. Moreover, suppose $n_0$ is the least integer with this property. If $n_0=0$, then
\[x\sim_\xi(\xi^{1-j}(z_1),\xi^{1-j}(z_2),\xi^{1-j}(z_3),\ldots).\]
If $n_0\geq1$ and $x_{n_0}$ is not in $\xi^{1-j}(F)$, then
\[x\sim_\xi(x_1,x_2,\ldots,x_{n_0},\xi^{1-j}(z_{n_0+1}),\xi^{1-j}(z_{n_0+2}),\ldots).\]
If $n_0\geq1$ and there is some $f$ in $F$ such that $x_{n_0}=\xi^{1-j}(f)$, then
\[x\sim_\xi(x_1,x_2,\ldots,x_{n_0-1},\xi^j(f),\xi^{1-j}(z_{n_0+1}),\xi^{1-j}(z_{n_0+2}),\ldots).\]
If none of the above occurs for $x$, then $x\sim_\xi x$ only. Each equivalence class thus consists of either one or two points. If $x\sim_\xi x'$ and $x\neq x'$, we will refer to the edge level $E_m$, where $m$ is the least integer such that $x_m\neq x_m'$, as the \textit{splitting level} of the equivalence class $\{x,x'\}$. We denote the quotient space $X_E/\sim_\xi$ by $X_\xi$ and let $\rho:X_E\to X_\xi$ denote the quotient map.

\begin{prop}\label{prod}
If $E=\xi^0(F)\cup\xi^1(F)$, then $X_E$ is homeomorphic to $X_F\times\{0,1\}^\bN$, and $X_\xi$ is homeomorphic to $X_F\times\bT$, where $\bT$ is the unit circle $\{z\in\bC\mid |z|=1\}$.
\end{prop}

\begin{proof}
Define the map $\phi:\{0,1\}^\bN\to \bT$ by $\phi(\{j_n\})=\exp\left(2\pi i\sum_{n=1}^\infty j_n2^{-n}\right)$, which is essentially the Cantor ternary function in a rather uncommon guise. Notice that two distinct sequences in $\{0,1\}^\bN$ have the same image under $\phi$ if and only if they are $(0,0,0,\ldots)$ and $(1,1,1,\ldots)$, or of the form
\[(j_1,j_2,\ldots,j_m,1,0,0,0,0,\ldots)\]
\[(j_1,j_2,\ldots,j_m,0,1,1,1,1,\ldots)\]
By our assumption and properties (iv) and (vi) of the embeddings, each $x$ in $X_E$ is uniquely determined by a path $z_x=(z_1,z_2,\ldots)$ in $X_F$ and a sequence $j_x=\{j_n\}$ in $\{0,1\}^\bN$ with $x_n=\xi^{j_n}(z_n)$ for all $n\geq1$. This gives an obvious continuous bijection from $X_E$ to $X_F\times\{0,1\}^\bN$, hence a homeomorphism. The composition of this homeomorphism with $\id_{X_F}\times\phi$ is a quotient map whose fibres are identical to the fibres of $\rho$.
\end{proof}

If the inclusion $E\supseteq\xi^0(F)\cup\xi^1(F)$ is proper, describing the space $X_\xi$ is slightly more complicated. In the upcoming lemma we describe some open sets in $X_E$ that remain open when passed to $X_\xi$.

\begin{definition}
Let $p=(p_1,p_2,\ldots,p_n)$ be a finite path in $(V,E)$. Denote
\[D_j(p)=\{x\in C(p)\mid x_k\in\xi^j(F)\FORAL k\geq n+1\}\]
for $j=0,1$.
\end{definition}

Recall that if $f:X\to Y$ is surjective, a subset $V$ of $X$ is called saturated if there is a subset $U$ of $Y$ with $V=f^{-1}(U)$.

\begin{lemma}\label{sat}
Let $p=(p_1,p_2,\ldots,p_n)$ be a finite path in $(V,E)$.
\begin{enumerate}[(i)]
    \item If $p_n$ is not in $\xi^0(F)\cup\xi^1(F)$, then $C(p)$ is clopen and saturated.
    \item If $p_n$ is in $\xi^0(F)\cup\xi^1(F)$, then
    \[U_p=C(p)-(D_0(p)\cup D_1(p))\]
    is open and saturated.
    \item Suppose that $k\geq m+1$ and $p$ and $q$ are two finite paths either of the form
    \[p=(\xi^0(z_1),\xi^0(z_2),\ldots,\xi^0(z_k))\]\[q=(\xi^1(z_1),\xi^1(z_2),\ldots,\xi^1(z_k))\]
    or of the form
    \[p=(x_1,x_2,\ldots,x_{m-1},x_m,\xi^0(z_{m+1}),\xi^0(z_{n+2}),\ldots,\xi^0(z_k))\]\[q=(x_1,x_2,\ldots,x_{m-1},x_m',\xi^1(z_{m+1}),\xi^1(z_{n+2}),\ldots,\xi^1(z_k))\]
    where $x_m$ is not in $\xi^0(F)$, $x_m'$ is not in $\xi^1(F)$, $x_m=x_m'$ if $x_m$ is not in $\xi^1(F)$, and $x_m'=\xi^0(f)$ if $x_m=\xi^1(f)$ for some $f$ in $F$. Then
    \[V_{p,q}=(C(p)-D_1(p))\cup(C(q)-D_0(q))\]
    is open and saturated.
\end{enumerate}
\end{lemma}

\begin{proof}
\begin{enumerate}[(i)]
    \item If $x$ is in $C(p)$ and $x\sim_\xi x'$, then the splitting level of $\{x,x'\}$ must be past level $E_n$ since $p_n$ is not in $\xi^0(F)\cup\xi^1(F)$, so $x'$ is in $C(p)$.
    \item If $x$ is in $U_p$ and $x\sim_\xi x'$, the removal of the sets $D_j(p)$ for $j=0,1$ forces the splitting level of $\{x,x'\}$ to be past level $E_n$, so $x$ and $x'$ coincide on the first $n$ edges and $x'$ is thus in $U_p$. It is straightforward to check that the sets $D_j(p)$ are closed, so $U_p$ is open.
    \item The proof that $V_{p,q}$ is open is analogous to that for (ii). If $x\in C(p)-D_1(p)$, $x\sim_\xi x'$, and the splitting level of $\{x,x'\}$ is past level $k$, then $x'$ is in $C(p)-D_1(p)$ similarly as in (ii). Otherwise, the splitting level must be either $E_m$ or $E_{m+1}$, but then $x'$ is in $C(q)-D_0(q)$.\qedhere
\end{enumerate}
\end{proof}

\begin{definition}
\begin{enumerate}[(i)]
    \item Let $\P$ denote all finite paths $p=(p_1,p_2,\ldots,p_n)$ in $(V,E)$ such that $p_n$ is not in $\xi^0(F)\cup\xi^1(F)$ but $t(p_n)$ is in $\xi(W)$.
    \item For $p$ in $\P$, define
    \[C_\xi(p)=\{x\in C(p)\mid x_k\in\xi^0(F)\cup\xi^1(F)\FORAL k\geq n+1\}\]
    To simplify notation, we will assume that there is an ``empty path" $p_0$ in $\P$ with the property that $C_\xi(p_0)=\{x\in X_E\mid x_k\in\xi^0(F)\cup\xi^1(F)\FORAL k\geq1\}$.
    \item For $w\in W_n$, define
\[X_F^{(w)}=\{z=(z_{n+1},z_{n+2},\ldots)\mid z_k\in F_k,t(z_k)=i(z_{k+1})\FORAL k\geq n+1,\AND i(z_{n+1})=w\}\]
that is, all infinite paths in $(W,F)$ that start at $w$.
\end{enumerate}
\end{definition}

Note that $C_\xi(p)$ is not equal to $D_0(p)\cup D_1(p)$, since edges in paths in $C_\xi(p)$ may alternate between the two embeddings. Regarding (iii) above, by fixing any finite path $p=(p_1,p_2,\ldots,p_n)$ in $(W,F)$ with $t(p_n)=w$, we may identify $X_F^{(w)}$ with $C(p)$ by associating $z$ with $pz$ (the concatenation of $p$ and $z$) and endow $X_F^{(w)}$ with the metric it receives from this identification. This makes $X_F^{(w)}$ a totally disconnected compact metric space.

\begin{lemma}\label{comp}
Let $p=(p_1,p_2,\ldots,p_n)$ be a finite path in $\P$ with $t(p_n)=\xi(w)$.
\begin{enumerate}[(i)]
    \item $C_\xi(p)$ is closed and saturated.
    \item $C_\xi(p)$ is homeomorphic to $X_F^{(w)}\times\{0,1\}^\bN$ and $\rho(C_\xi(p))$ is homeomorphic to $X_F^{(w)}\times \bT$.
    \item If $p\neq q$ in $\P$, there are disjoint clopen saturated sets $U$ and $V$ in $X_E$ such that $C_\xi(p)\subseteq U$ and $C_\xi(q)\subseteq V$; in particular, $\rho(C_\xi(p))$ and $\rho(C_\xi(q))$ are disjoint.
    \item For every $x$ in $X_E$ and $\delta>0$, there is an open saturated subset $U$ of $X_E$ such that $\rho^{-1}(\rho(x))\subseteq U\subseteq B(\rho^{-1}(\rho(x)),\delta)$.
    \item If $x$ is in $X_E-\bigcup_{p\in\P}C_\xi(p)$, then $U$ from part (iv) can be chosen to be clopen; it follows that $\{\rho(x)\}$ is a connected component of $X_\xi$ in this case.
\end{enumerate}
\end{lemma}

\begin{proof}
\begin{enumerate}[(i)]
    \item For $m\geq n+1$, let $A_m$ be the finite union of all cylinder sets of the form
    \[C(p_1,p_2,\ldots,p_n,q_{n+1},q_{n+2},\ldots,q_m)\]
    where $q_k$ is in $\xi^0(F)\cup\xi^1(F)$ for $n+1\leq k\leq m$. Each $A_m$ is closed and $C_\xi(p)=\bigcap_{m=n+1}^\infty A_m$. The set $C_\xi(p)$ is saturated since $p_n$ is not in $\xi^0(F)\cup\xi^1(F)$.
    \item Each $x$ in $C_\xi(p)$ is uniquely determined by a path in $X_F^{(w)}$ and a sequence in $\{0,1\}^\bN$, and the proof then proceeds analogously to that of Proposition \ref{prod}.
    \item If $C(p)\cap C(q)=\mt$, then $U=C(p)$ and $V=C(q)$ suffice. If $C(p)\subseteq C(q)$, take $U=C(p)$ and $V=X_E-C(p)$.
    \item We consider two cases.
    \begin{enumerate}
        \item Suppose that $\rho^{-1}(\rho(x))=\{x\}$. Choose $k\geq1$ such that $2^{-k}<\delta$, let $p=(x_1,\ldots,x_k)$, and let $U=U_p$ as in part (ii) of Lemma \ref{sat}. Observe that $x$ is not in $D_0(p)\cup D_1(p)$ because otherwise $\rho^{-1}(\rho(x))$ would consist of two points.
        \item If $\rho^{-1}(\rho(x))=\{x,x'\}$ and $E_m$ is the splitting level, choose $k\geq m$ with $2^{-k}<\delta$, let $p=(x_1,\ldots,x_k)$ and $q=(x_1',\ldots,x_k')$, and let $U=V_{p,q}$ from part (iii) of Lemma \ref{sat}.
    \end{enumerate}
    \item There are infintely many edges $x_{n_1},x_{n_2},\ldots$ of $x$ that are not in $\xi^0(F)\cup\xi^1(F)$. Choose $k\geq1$ such that $2^{-n_k}<\delta$, let $U=C(x_1,\ldots,x_{n_k})$, and use part (i) of Lemma \ref{sat}.\qedhere
\end{enumerate}
\end{proof}

\begin{prop}\label{metric}
\begin{enumerate}[(i)]
    \item The quotient map $\rho:X_E\to X_\xi$ is closed.
    \item The space $X_\xi$ is second-countable, compact, and Hausdorff, hence metrizable.
    \item Every connected component of $X_\xi$ is either a single point or homeomorphic to $\bT$.
    \item The covering dimension of $X_\xi$ is $1$.
\end{enumerate}
\end{prop}

\begin{proof}
\begin{enumerate}[(i)]
    \item Let $F$ be closed in $X_E$ and $y$ a point not in $\rho(F)$. Then $F\cap\rho^{-1}(y)=\mt$, so $F\cap B(\rho^{-1}(y),\delta)=\mt$ for some $\delta>0$. By part (iv) of Lemma \ref{comp}, we can find an open saturated set $U$ of $X_E$ with $\rho^{-1}(y)\subseteq U$ and $F\cap U=\mt$. Then $\rho(U)$ is a neighbourhood of $y$ disjoint from $\rho(F)$.
    \item A closed continuous surjective map with compact fibres preserves the Hausdorff property and the second-countable property, see Section 31, Exercise 7 of \cite{munkres}. Metrizability follows from the Urysohn metrization theorem.
    \item Let $Y$ be a connected component of $X_\xi$. If $Y$ consists of more than one point, it must be contained in $\bigcup_{p\in\P}\rho(C_\xi(p))$ by part (v) of Lemma \ref{comp}. By part (iii) of Lemma \ref{comp}, it must be contained in a single $\rho(C_\xi(p))$. The conclusion then follows from part (ii) of Lemma \ref{comp} since $X_F^{(w)}$ is totally disconnected.
    \item By part (ii) there is a metric $d$ which gives the topology of $X_\xi$. The space $X_\xi$ contains at least one homeomorphic copy of $\bT$, so it suffices to prove that $\dim X_\xi\leq1$. We can show this by showing that, for any $\ep>0$, we can find two collections $\U_0$ and $\U_1$ of pairwise disjoint open saturated subsets of $X_E$ which, when taken together, cover $X_E$ and $\diam\rho(U)<\ep$ for all $U$ in $\U_0\cup\U_1$. Choose $n\geq1$ such that $d_E(x,y)<2^{-n}$ implies that $d(\rho(x),\rho(y))<\frac\ep2$, and consider a finite path $p=(p_1,\ldots,p_n)$. If $p_n$ is not in $\xi^0(F)\cup\xi^1(F)$, put $C(p)$ in $\U_0$. Otherwise, put the set $U_p$ from part (ii) of Lemma \ref{sat} in $\U_0$. By the definition of the sets $U_p$, the remaining paths possibly not covered by the elements of $\U_0$ are the two-point equivalence classes whose splitting levels are among the levels $E_1,E_2,\ldots,E_n$. To cover these, let $\U_1$ consist of all sets of the form $V_{p,q}$ from part (iii) of Lemma \ref{sat}, where $k=n$ and $1\leq m\leq n-1$.\qedhere
\end{enumerate}
\end{proof}

\begin{example}\label{simp}
Let $(V,E)$ be the Bratteli diagram with $\#V_n=1$ for all $n$ and $\#E_n=2$ for all $n$ (the $2^\infty$ diagram), while $(W,F)$ is the diagram with $\#W_n=\#F_n=1$ for all $n$. Embed one copy of $(W,F)$ on the left and the other copy on the right. Then $X_\xi$ is homeomorphic to $\bT$ by Proposition \ref{prod}.
\end{example}

\begin{example}\label{cantor}
Let $(V,E)$ be the Bratteli diagram with $\#V_n=1$ for all $n$ and $\#E_n=4$ for all $n$ (the $4^\infty$ diagram), while $(W,F)$ is the diagram with $\#W_n=1$ and $\#F_n=2$ for all $n$. Embed one copy of $(W,F)$ on the left and the other copy on the right. Then $X_\xi$ is homeomorphic to $X_F\times \bT$ by Proposition \ref{prod}. By identifying $X_F$ with the middle thirds Cantor set $X\subseteq[1,2]$ (shifted to the right from its usual position), we can identify $X_\xi$ with the planar set $\bigcup_{x\in X}x\bT$, where $x\bT=\{xz\mid z\in \bT\}$.
\end{example}

\begin{example}\label{2}
Consider the Bratteli diagram in Figure \ref{1/n}. Let $(W,F)$ be the diagram with $\#W_n=\#F_n=1$ for all $n$. Embed $(W,F)$ down the bottom edges of $(V,E)$. Then $X_\xi$ is homeomorphic to $(\bigcup_{n=0}^\infty2^{-n}\bT)\cup\{0\}$. Indeed, any finite path ending in a diagonal edge is in $\P$, and $\rho(C_\xi(p))$ is clopen and homeomorphic to $\bT$ for each such path. The only path not in $\bigcup_{p\in\P}C_\xi(p)$ is the path along the top of the diagram, which corresponds to the point $0$.
\begin{figure}[ht]
    \begin{tikzpicture}
     \filldraw[black] (0,0) circle (2pt);
     \filldraw[black] (2,1) circle (2pt);
     \filldraw[black] (2,-1) circle (2pt);
     \filldraw[black] (4,1) circle (2pt);
     \filldraw[black] (4,-1) circle (2pt);
     \filldraw[black] (6,1) circle (2pt);
     \filldraw[black] (6,-1) circle (2pt);
     \filldraw[black] (6.7,0) circle (1pt);
     \filldraw[black] (7,0) circle (1pt);
     \filldraw[black] (7.3,0) circle (1pt);
     \draw[black,thick] (0,0) -- (2,1);
     \draw[black,thick] (2,-1) .. controls (3,-1.25) .. (4,-1);
     \draw[black,thick] (0,0) .. controls (0.9,-0.7) .. (2,-1);
     \draw[black,thick] (0,0) .. controls (1.1,-0.3) .. (2,-1);
     \draw[black,thick] (2,1) -- (4,1);
     \draw[black,thick] (4,1) -- (6,1);
     \draw[black,thick] (6,1) -- (6.3,0.75);
     \draw[black,thick] (6,1) -- (6.3,1);
     \draw[black,thick] (6,-1) -- (6.3,-0.9);
     \draw[black,thick] (6,-1) -- (6.3,-1.1);
     \draw[black,thick] (2,1) -- (4,-1);
     \draw[black,thick] (4,1) -- (6,-1);
     \draw[black,thick] (4,-1) .. controls (5,-0.75) .. (6,-1);
     \draw[black,thick] (4,-1) .. controls (5,-1.25) .. (6,-1);
     \draw[black,thick] (2,-1) .. controls (3,-0.75) .. (4,-1);
     \end{tikzpicture}
    \caption{The Bratteli diagram $(V,E)$ in Example \ref{2}}
    \label{1/n}
\end{figure}
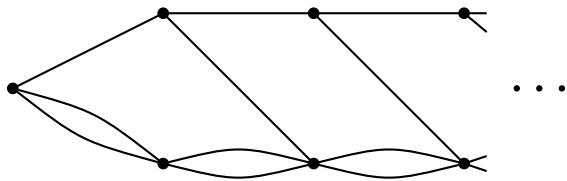
\end{example}

\begin{example}\label{3}
Let $(V,E)$ be the Bratteli diagram with $\#V_n=1$ for all $n$ and $\#E_n=3$ for all $n$ (the $3^\infty$ diagram), while $(W,F)$ is the diagram with $\#W_n=\#F_n=1$ for all $n$. For every finite path $p=(p_1,p_2,\ldots,p_n)$ in $\P$, $\rho(C_\xi(p))$ is homeomorphic to $\bT$ by part (ii) of Lemma \ref{comp}. To help illustrate $X_\xi$, we note that $\bigcup_{p\in\P}C_\xi(p)$ is dense in $X_E$ and construct a continuous map from $\bigcup_{p\in\P}C_\xi(p)$ to the plane $\bC$ which is injective except on the fibres of $\rho$. We will identify $X_E$ with $\{0,1,2\}^\bN$ and regard the edges labelled with $j$ the edges in $\xi^j(F)$ for $j=0,1$.

First define $\theta_2=0$ and for a finite path $p=(p_1,p_2,\ldots,p_n)$ in $\{0,1\}^{n-1}\times\{2\}$ for $n\geq2$, set $\theta_p=2^{-n}+\sum_{k=1}^{n-1}p_k2^{-k}$. Define the function $f_p:\bC\to\bC$ by
\[f_p(z)=e^{2\pi i\theta_p}(2^{-n}z+(1+2^{-(n-1)}))\]
Now suppose that $p$ in $\P$ is arbitrary, and partition $p=p^{(1)}p^{(2)}\cdots p^{(m)}$ where each $p^{(k)}$ is in $\{0,1\}^{l_k}\times\{2\}$ for some integer $l_k\geq0$ (note that such a partition exists and is unique). Define
\[f_p=f_{p^{(1)}}\circ f_{p^{(2)}}\circ\cdots\circ f_{p^{(m)}}\]
If $p_0$ is the empty path, let $f_{p_0}(z)=z$. If $x=(x_1,x_2,\ldots)$ is in $C_\xi(p)$, define
\[\rho(x)=f_p\left(\exp\left(2\pi i\sum_{k=n+1}^\infty x_k2^{-k}\right)\right)\]
Letting $X_n=\bigcup_{\text{length}(p)\leq n-1}C_\xi(p)$ (where $\text{length}(p)$ is the number of edges in $p$ with the convention that $\text{length}(p_0)=0$), we have $\rho(X_1)$ the unit circle and $\rho(X_2)$ is the disjoint union of the unit circle and the circle of radius $\frac12$ centred at $2$. Figure \ref{circlepictures} illustrates $\rho(X_n)$ for $n=3,4,5,6$.

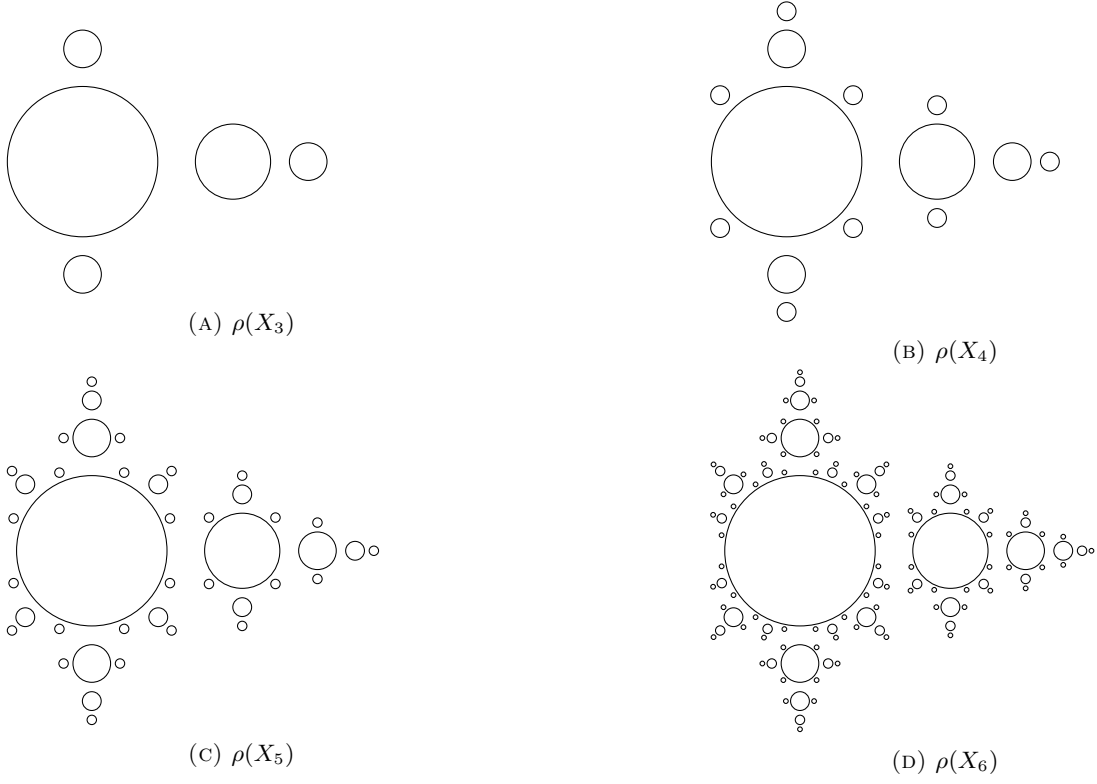
\begin{figure}[ht]

\begin{subfigure}{0.4\textwidth}
\begin{tikzpicture}
%\draw[step=1cm,gray,very thin] (-3,-6) grid (8,6);
%\foreach \x in {0,1,2,3,4,5,6,7}
 %  \draw (\x cm,1pt) -- (\x cm,-1pt) node[anchor=north] {$\x$};
%\foreach \y in {0,1,2,3,4,5}
 %   \draw (1pt,\y cm) -- (-1pt,\y cm) node[anchor=east] {$\y$};
\draw(0,0) circle (1cm);%1
\draw(2,0) circle (0.5cm);%2
\draw(0,1.5) circle (0.25cm);%3
\draw(0,-1.5) circle (0.25cm);%3
\draw(3,0) circle (0.25cm);%3
\end{tikzpicture}
\caption{$\rho(X_3)$}
\end{subfigure}
\hfill
\begin{subfigure}{0.4\textwidth}
\begin{tikzpicture}
%\draw[step=1cm,gray,very thin] (-3,-6) grid (8,6);
%\foreach \x in {0,1,2,3,4,5,6,7}
 %  \draw (\x cm,1pt) -- (\x cm,-1pt) node[anchor=north] {$\x$};
%\foreach \y in {0,1,2,3,4,5}
 %   \draw (1pt,\y cm) -- (-1pt,\y cm) node[anchor=east] {$\y$};
\draw(0,0) circle (1cm);%1
\draw(2,0) circle (0.5cm);%2
\draw(0,1.5) circle (0.25cm);%3
\draw(0,-1.5) circle (0.25cm);%3
\draw(3,0) circle (0.25cm);%3
\draw(-0.8838834765,0.8838834765) circle (0.125cm);%4
\draw(-0.8838834765,-0.8838834765) circle (0.125cm);%4
\draw(0.8838834765,0.8838834765) circle (0.125cm);%4
\draw(0.8838834765,-0.8838834765) circle (0.125cm);%4
\draw(0,2) circle (0.125cm);%4
\draw(0,-2) circle (0.125cm);%4
\draw(2,-0.75) circle (0.125cm);%4
\draw(2,0.75) circle (0.125cm);%4
\draw(3.5,0) circle (0.125cm);%4
\end{tikzpicture}
\caption{$\rho(X_4)$}
\end{subfigure}

\begin{subfigure}{0.4\textwidth}
\begin{tikzpicture}
%\draw[step=1cm,gray,very thin] (-3,-6) grid (8,6);
%\foreach \x in {0,1,2,3,4,5,6,7}
 %  \draw (\x cm,1pt) -- (\x cm,-1pt) node[anchor=north] {$\x$};
%\foreach \y in {0,1,2,3,4,5}
 %   \draw (1pt,\y cm) -- (-1pt,\y cm) node[anchor=east] {$\y$};
\draw(0,0) circle (1cm);%1
\draw(2,0) circle (0.5cm);%2
\draw(0,1.5) circle (0.25cm);%3
\draw(0,-1.5) circle (0.25cm);%3
\draw(3,0) circle (0.25cm);%3
\draw(-0.8838834765,0.8838834765) circle (0.125cm);%4
\draw(-0.8838834765,-0.8838834765) circle (0.125cm);%4
\draw(0.8838834765,0.8838834765) circle (0.125cm);%4
\draw(0.8838834765,-0.8838834765) circle (0.125cm);%4
\draw(0,2) circle (0.125cm);%4
\draw(0,-2) circle (0.125cm);%4
\draw(2,-0.75) circle (0.125cm);%4
\draw(2,0.75) circle (0.125cm);%4
\draw(3.5,0) circle (0.125cm);%4
\draw(3.75,0) circle (0.0625cm);%5
\draw(3,0.375) circle (0.0625cm);%5
\draw(3,-0.375) circle (0.0625cm);%5
\draw(2.4419417383,-0.4419417383) circle (0.0625cm);%5
\draw(2.4419417383,0.4419417383) circle (0.0625cm);%5
\draw(1.5580582618,-0.4419417383) circle (0.0625cm);%5
\draw(1.5580582618,0.4419417383) circle (0.0625cm);%5
\draw(2,1) circle (0.0625cm);%5
\draw(2,-1) circle (0.0625cm);%5
\draw(0,2.25) circle (0.0625cm);%5
\draw(0,-2.25) circle (0.0625cm);%5
\draw(-0.375,-1.5) circle (0.0625cm);%5
\draw(-0.375,1.5) circle (0.0625cm);%5
\draw(0.375,-1.5) circle (0.0625cm);%5
\draw(0.375,1.5) circle (0.0625cm);%5
\draw(1.0606601718,1.0606601718) circle (0.0625cm);%5
\draw(-1.0606601718,1.0606601718) circle (0.0625cm);%5
\draw(1.0606601718,-1.0606601718) circle (0.0625cm);%5
\draw(-1.0606601718,-1.0606601718) circle (0.0625cm);%5
\draw(1.0393644741,0.4305188614) circle (0.0625cm);%5
\draw(1.0393644741,-0.4305188614) circle (0.0625cm);%5
\draw(-1.0393644741,0.4305188614) circle (0.0625cm);%5
\draw(-1.0393644741,-0.4305188614) circle (0.0625cm);%5
\draw(0.4305188614,1.0393644741) circle (0.0625cm);%5
\draw(0.4305188614,-1.0393644741) circle (0.0625cm);%5
\draw(-0.4305188614,1.0393644741) circle (0.0625cm);%5
\draw(-0.4305188614,-1.0393644741) circle (0.0625cm);%5
\end{tikzpicture}
\caption{$\rho(X_5)$}
\end{subfigure}
\hfill
\begin{subfigure}{0.4\textwidth}
\begin{tikzpicture}
 %\draw[step=1cm,gray,very thin] (-3,-6) grid (8,6);
%\foreach \x in {0,1,2,3,4,5,6,7}
 %  \draw (\x cm,1pt) -- (\x cm,-1pt) node[anchor=north] {$\x$};
%\foreach \y in {0,1,2,3,4,5}\
 %   \draw (1pt,\y cm) -- (-1pt,\y cm) node[anchor=east] { $\y$};
\draw(0,0) circle (1cm);%1
\draw(2,0) circle (0.5cm);%2
\draw(0,1.5) circle (0.25cm);%3
\draw(0,-1.5) circle (0.25cm);%3
\draw(3,0) circle (0.25cm);%3
\draw(-0.8838834765,0.8838834765) circle (0.125cm);%4
\draw(-0.8838834765,-0.8838834765) circle (0.125cm);%4
\draw(0.8838834765,0.8838834765) circle (0.125cm);%4
\draw(0.8838834765,-0.8838834765) circle (0.125cm);%4
\draw(0,2) circle (0.125cm);%4
\draw(0,-2) circle (0.125cm);%4
\draw(2,-0.75) circle (0.125cm);%4
\draw(2,0.75) circle (0.125cm);%4
\draw(3.5,0) circle (0.125cm);%4
\draw(3.75,0) circle (0.0625cm);%5
\draw(3,0.375) circle (0.0625cm);%5
\draw(3,-0.375) circle (0.0625cm);%5
\draw(2.4419417383,-0.4419417383) circle (0.0625cm);%5
\draw(2.4419417383,0.4419417383) circle (0.0625cm);%5
\draw(1.5580582618,-0.4419417383) circle (0.0625cm);%5
\draw(1.5580582618,0.4419417383) circle (0.0625cm);%5
\draw(2,1) circle (0.0625cm);%5
\draw(2,-1) circle (0.0625cm);%5
\draw(0,2.25) circle (0.0625cm);%5
\draw(0,-2.25) circle (0.0625cm);%5
\draw(0.375,-1.5) circle (0.0625cm);%5
\draw(0.375,1.5) circle (0.0625cm);%5
\draw(-0.375,-1.5) circle (0.0625cm);%5
\draw(-0.375,1.5) circle (0.0625cm);%5
\draw(1.0606601718,1.0606601718) circle (0.0625cm);%5
\draw(-1.0606601718,1.0606601718) circle (0.0625cm);%5
\draw(1.0606601718,-1.0606601718) circle (0.0625cm);%5
\draw(-1.0606601718,-1.0606601718) circle (0.0625cm);%5
\draw(1.0393644741,0.4305188614) circle (0.0625cm);%5
\draw(1.0393644741,-0.4305188614) circle (0.0625cm);%5
\draw(-1.0393644741,0.4305188614) circle (0.0625cm);%5
\draw(-1.0393644741,-0.4305188614) circle (0.0625cm);%5
\draw(0.4305188614,1.0393644741) circle (0.0625cm);%5
\draw(0.4305188614,-1.0393644741) circle (0.0625cm);%5
\draw(-0.4305188614,1.0393644741) circle (0.0625cm);%5
\draw(-0.4305188614,-1.0393644741) circle (0.0625cm);%5
\draw(3.875,0) circle (0.03125cm);%6
\draw(3.2209708691,0.2209708691) circle (0.03125cm);%6
\draw(3.2209708691,-0.2209708691) circle (0.03125cm);%6
\draw(3,0.5) circle (0.03125cm);%6
\draw(3,-0.5) circle (0.03125cm);%6
\draw(2.7790291309,0.2209708691) circle (0.03125cm);%6
\draw(2.7790291309,-0.2209708691) circle (0.03125cm);%6
\draw(3.5,0.1875) circle (0.03125cm);%6
\draw(3.5,-0.1875) circle (0.03125cm);%6
\draw(2,1.125) circle (0.03125cm);%6
\draw(2,-1.125) circle (0.03125cm);%6
\draw(1.8125,0.75) circle (0.03125cm);%6
\draw(2.1875,0.75) circle (0.03125cm);%6
\draw(1.8125,-0.75) circle (0.03125cm);%6
\draw(2.1875,-0.75) circle (0.03125cm);%6
\draw(2.5303300859,0.5303300859) circle (0.03125cm);%6
\draw(2.5303300859,-0.5303300859) circle (0.03125cm);%6
\draw(1.4696699141,0.5303300859) circle (0.03125cm);%6
\draw(1.4696699141,-0.5303300859) circle (0.03125cm);%6
\draw(2.2152594307,0.5196822371) circle (0.03125cm);%6
\draw(1.7847405693,0.5196822371) circle (0.03125cm);%6
\draw(1.7847405693,-0.5196822371) circle (0.03125cm);%6
\draw(2.2152594307,-0.5196822371) circle (0.03125cm);%6
\draw(2.5196822371,0.2152594307) circle (0.03125cm);%6
\draw(2.5196822371,-0.2152594307) circle (0.03125cm);%6
\draw(1.480317763,0.2152594307) circle (0.03125cm);%6
\draw(1.480317763,-0.2152594307) circle (0.03125cm);%6
\draw(0,2.375) circle (0.03125cm);%6
\draw(0,-2.375) circle (0.03125cm);%6
\draw(0.1875,2) circle (0.03125cm);%6
\draw(0.1875,-2) circle (0.03125cm);%6
\draw(-0.1875,2) circle (0.03125cm);%6
\draw(-0.1875,-2) circle (0.03125cm);%6
\draw(0.5,1.5) circle (0.03125cm);%6
\draw(0.5,-1.5) circle (0.03125cm);%6
\draw(-0.5,1.5) circle (0.03125cm);%6
\draw(-0.5,-1.5) circle (0.03125cm);%6
\draw(0.2209708691,1.7209708691) circle (0.03125cm);%6
\draw(0.2209708691,1.2847405693) circle (0.03125cm);%6
\draw(-0.2209708691,1.7209708691) circle (0.03125cm);%6
\draw(-0.2209708691,1.2847405693) circle (0.03125cm);%6
\draw(0.2209708691,-1.7209708691) circle (0.03125cm);%6
\draw(-0.2209708691,-1.7209708691) circle (0.03125cm);%6
\draw(0.2209708691,-1.2847405693) circle (0.03125cm);%6
\draw(-0.2209708691,-1.2847405693) circle (0.03125cm);%6
\draw(1.1490485195,1.1490485195) circle (0.03125cm);%6
\draw(-1.1490485195,1.1490485195) circle (0.03125cm);%6
\draw(1.1490485195,-1.1490485195) circle (0.03125cm);%6
\draw(-1.1490485195,-1.1490485195) circle (0.03125cm);%6
\draw(1.0420843605,0.2072834672) circle (0.03125cm);%6
\draw(1.0420843605,-0.2072834672) circle (0.03125cm);%6
\draw(-1.0420843605,0.2072834672) circle (0.03125cm);%6
\draw(-1.0420843605,-0.2072834672) circle (0.03125cm);%6
\draw(0.2072834672,1.0420843605) circle (0.03125cm);%6
\draw(0.2072834672,-1.0420843605) circle (0.03125cm);%6
\draw(-0.2072834672,1.0420843605) circle (0.03125cm);%6
\draw(-0.2072834672,-1.0420843605) circle (0.03125cm);%6
\draw(0.5902933726,0.8834364631) circle (0.03125cm);%6
\draw(0.5902933726,-0.8834364631) circle (0.03125cm);%6
\draw(-0.5902933726,0.8834364631) circle (0.03125cm);%6
\draw(-0.5902933726,-0.8834364631) circle (0.03125cm);%6
\draw(0.8834364631,0.5902933726) circle (0.03125cm);%6
\draw(0.8834364631,-0.5902933726) circle (0.03125cm);%6
\draw(-0.8834364631,0.5902933726) circle (0.03125cm);%6
\draw(-0.8834364631,-0.5902933726) circle (0.03125cm);%6
\draw(0.4783542904,1.1548494157) circle (0.03125cm);%6
\draw(0.4783542904,-1.1548494157) circle (0.03125cm);%6
\draw(-0.4783542904,1.1548494157) circle (0.03125cm);%6
\draw(-0.4783542904,-1.1548494157) circle (0.03125cm);%6
\draw(1.1548494157,0.4783542904) circle (0.03125cm);%6
\draw(1.1548494157,-0.4783542904) circle (0.03125cm);%6
\draw(-1.1548494157,0.4783542904) circle (0.03125cm);%6
\draw(-1.1548494157,-0.4783542904) circle (0.03125cm);%6
\draw(1.016465998,0.751300955) circle (0.03125cm);%6
\draw(1.016465998,-0.751300955) circle (0.03125cm);%6
\draw(-1.016465998,0.751300955) circle (0.03125cm);%6
\draw(-1.016465998,-0.751300955) circle (0.03125cm);%6
\draw(0.751300955,1.016465998) circle (0.03125cm);%6
\draw(0.751300955,-1.016465998) circle (0.03125cm);%6
\draw(-0.751300955,1.016465998) circle (0.03125cm);%6
\draw(-0.751300955,-1.016465998) circle (0.03125cm);%6
\end{tikzpicture}
\caption{$\rho(X_6)$}
\end{subfigure}
\caption{Geometric approximations of $X_\xi$ from Example \ref{3}}
\label{circlepictures}

\end{figure}

\end{example}

We now turn to the factor groupoids $R_\xi$.

\begin{definition}
Let $R_\xi=\rho\times\rho(R_E)$ and endow it with the quotient topology induced by the \'etale topology on $R_E$. Denote $\pi:R_E\to R_\xi$ to be the restriction of $\rho\times\rho$ to $R_E$.
\end{definition}

Clearly $\pi$ is a continuous surjective groupoid homomorphism, and since tail-equivalent paths are never in the same fibre of $\rho$, $\pi$ is bijective on the fibres of the range map $r:R_E\to R_E^{(0)}$. We observe that the sets $R_n$ and $R_{n+1}-R_n$ are saturated with respect to $\pi$, for all $n\geq1$.

%We record a useful way of obtaining open sets in $R_\xi$ from the basic open sets $\gamma(p,q)$: if $t(p)=t(q)$ in $V_n$ is not in $\xi^0(W)=\xi^1(W)$, then $\pi(\gamma(p,q))$ is open. Otherwise, remove all pairs $(x,y)$ such that $x_k=y_k$ is in $\xi^0(F)$ for all $k\geq n+1$ and all pairs such that $x_k=y_k$ is in $\xi^1(F)$ for all $k\geq n+1$. All such pairs form a closed saturated set, and thus the resulting set is open and saturated.

\begin{prop}
\begin{enumerate}[(i)]
    \item $R_\xi$ is a second-countable, locally compact, Hausdorff, \'etale groupoid.
    \item The map $\pi:R_E\to R_\xi$ satisfies the hypotheses in Proposition \ref{factor}, and is regular and measure regular.
\end{enumerate}
\end{prop}

\begin{proof}
We begin by proving that $R_\xi$ is Hausdorff. Notice that $\pi$ is continuous with respect to the relative product topologies from $X_E\times X_E$ and $X_\xi\times X_\xi$. It follows that the quotient topology $R_\xi$ receives from $R_E$ is finer than the relative topology from $X_\xi\times X_\xi$. The space $X_\xi$ is Hausdorff by Proposition \ref{metric}, thus $R_\xi$ is too.

To see that $R_\xi$ is locally compact, notice that $R_E=\bigcup_{n=1}^\infty R_n$ and each $R_n$ is compact, open, and saturated. Therefore each element of $R_\xi$ is contained in the compact open set $\pi(R_n)$ for some $n$.

Let $\mu_E:R_E^{(2)}\to R_E$ and $\mu_\xi:R_\xi^{(2)}\to R_\xi$ be the multiplication maps. We will show that $\mu_\xi$ is continuous. It is readily verified that $\pi\circ\mu_E=\mu_\xi\circ(\pi\times\pi)$. If $U$ is open in $R_\xi$, then $\mu_E^{-1}(\pi^{-1}(U))=(\pi\times\pi)^{-1}(\mu_\xi^{-1}(U))$ is open in $R_E^{(2)}$ since $\mu_E$ and $\pi$ are continuous. Since $R_E$ and $R_\xi$ are locally compact Hausdorff and $R_E^{(2)}$ is a closed subset of $R_E\times R_E$ which is saturated (because $(\pi\times\pi)^{-1}(R_\xi^{(2)})=R_E^{(2)}$), the map $\pi\times\pi|_{R_E^{(2)}}$ is a quotient map, so $\mu_\xi^{-1}(U)=\pi\times\pi((\pi\times\pi)^{-1}(\mu_\xi^{-1}(U)))$ is open in $R_\xi^{(2)}$. Continuity of inversion may be proven similarly.

To see that $R_\xi$ is \'etale, first observe that, since $R_E^{(0)}=\pi^{-1}(R_\xi^{(0)})$, the groupoid is $r$-discrete (that is, $R_\xi^{(0)}$ is open), so it suffices to show that $r$ is an open map. Let $U'\subseteq R_\xi$ be open. Then $U=\pi^{-1}(U')$ is open by continuity of $\pi$, and $r(U)$ is open since $R_E$ is \'etale. Now $\pi(r(U))=r(\pi(U))=r(U')$, so we need only show that $r(U)$ is saturated with respect to $\pi$. If $(x,x)$ is in $r(U)$ and $(x',x')$ is in $R_E^{(0)}$ with $\rho(x)=\rho(x')$, pick $y$ in $X_E$ with $(x,y)$ in $U$. Then there is some $y'$ in $X_E$ with $\rho(y)=\rho(y')$. Then $(x',y')$ is in $R_E$ and $U$ is saturated, so $(x',y')$ is in $U$. Then $(x',x')=r(x',y')$ is in $r(U)$.

Of the standing hypotheses, the only nontrivial one to check is that $\pi$ is proper. Let $K\subseteq R_\xi$ be compact. As noted earlier, $R_E$ is the union of the compact open saturated sets $R_n$. It follows that $K\subseteq\pi(R_n)$ for some $n$, so $\pi^{-1}(K)\subseteq R_n$. Being a closed subset of a compact set, $\pi^{-1}(K)$ is compact.

At this point, we may conclude that $\pi$ is a closed map, and hence that $R_\xi$ is second-countable, using Lemma 7.5 of \cite{exc} and the proof of parts (i) and (ii) of Proposition \ref{metric}.

Next, we show that $\pi$ is regular. Let $(x,y)$ and $(x',y')$ be two distinct pairs in $R_E$ with $\pi(x,y)=\pi(x',y')$ and $\ep>0$. Let $n\geq0$ be the integer such that $(x,y)$ and $(x',y')$ are in $R_{n+1}-R_n$, let $E_m$ be the splitting level of $\{x,x'\}$, and let $E_l$ be the splitting level of $\{y,y'\}$. As before, we assume that the edges of $x$ and $y$ are eventually in $\xi^0(F)$ and the edges of $x'$ and $y'$ are eventually in $\xi^1(F)$. Choose $k\geq\max\{l,m,n+1\}$ so that $2^{-k}<\ep$. Let $p=(x_1,\ldots,x_k)$, $p'=(x_1',\ldots,x_k')$, $q=(y_1,\ldots,y_k)$, and $q'=(y_1',\ldots,y_k')$, and let
\[U=(\gamma(p,q)\cup\gamma(p',q'))\cap(V_{p,p'}\times V_{q,q'})\]
We must show that
\begin{enumerate}[(i)]
    \item $U$ is open,
    \item $U$ is saturated, and
    \item $\pi(U)$ satisfies the definition of regularity.
\end{enumerate}
For (i), since the \'etale topology on $R_E$ is finer than the relative product topology, $V_{p,p'}\times V_{q,q'}$ is open in $R_E$, and therefore so is $U$.

(ii) and (iii) can be proved simultaneously, in two cases. Let $(x'',y'')$ be in $U$ (suppose, without loss of generality, that it is in $\gamma(p,q)$) and $(x''',y''')$ be another pair with $\pi(x'',y'')=\pi(x''',y''')$. We must show that $(x''',y''')$ is also in $U$ and that either
\begin{equation}\label{hausdorff}
d_E^{(2)}(\{(x,y),(x',y')\},\{(x'',y''),(x''',y''')\})<\ep
\end{equation}
(in the Hausdorff metric) or
\begin{equation}\label{close}
d_E^{(2)}((x'',y''),(x''',y'''))<\ep
\end{equation}

The first case considers when $x_j''=y_j''$ is in $\xi^0(F)$ for all $j\geq k+1$. Here, the splitting level of $\{x'',x'''\}$ must be $E_m$ and the splitting level of $\{y'',y'''\}$ must be $E_l$. This forces $(x''',y''')$ into $\gamma(p',q')\cap(V_{p,p'}\times V_{q,q'})\subseteq U$. Also, we have
\begin{align}
\nonumber d_E^{(2)}(\{(x,y),(x',y')\},\{(x'',y''),(x''',y''')\})&=\max\{d_E^{(2)}((x,y),(x'',y'')),d_E^{(2)}((x',y'),(x''',y'''))\}\\
\nonumber&=\max\{d_E(x,x''),d_E(y,y''),d_E(x',x'''),d_E(y',y''')\}\\
\nonumber&\leq2^{-k}\\
\nonumber&<\ep
\end{align}
and hence (\ref{hausdorff}) holds.

The second case considers when $x_j''=y_j''$ is not in $\xi^0(F)$ for some $j\geq k+1$. This assumption, together with $(x'',y'')$ being in $V_{p,p'}\times V_{q,q'}$, forces the splitting levels of $\{x'',x'''\}$ and $\{y'',y'''\}$ to be level $E_{k+1}$ or further. In this case, $(x''',y''')$ is also in $\gamma(p,q)\cap(V_{p,p'}\times V_{q,q'})\subseteq U$.
Moreover,
\begin{align}
\nonumber d_E^{(2)}((x'',y''),(x''',y'''))&=\max\{d_E(x'',x'''),d_E(y'',y''')\}\\
\nonumber&\leq2^{-k}\\
\nonumber&<\ep
\end{align}
and hence (\ref{close}) holds. Finally, $\pi$ is measure regular by Proposition 7.16 of \cite{exc}.
\end{proof}

\begin{prop}
The dynamic asymptotic dimension of $R_\xi$ is $0$. It follows that $R_\xi$ is amenable, and that the nuclear dimension of $C_r^*(R_\xi)$ is at most $1$.
\end{prop}

\begin{proof}
Since $R_\xi$ is equal to the union of the compact open subgroupoids $\pi(R_n)$, its dynamic asymptotic dimension is zero by Example 5.3 of \cite{dad}. The conclusion now follows from part (iv) of Proposition \ref{metric}, Theorem 8.6 of \cite{dad}, and Corollary 8.25 of \cite{dad}.
\end{proof}

All in all, the algebras $C_r^*(R_\xi)$ are unital, separable, have finite nuclear dimension, and satisfy the UCT (Theorem 10.1.7 of \cite{groupoid}). If, in addition, $R_E$ is minimal, then so is $R_\xi$ because $\pi$ is continuous and surjective, hence dense equivalence classes map to dense equivalence classes. In this case, $C_r^*(R_\xi)$ is simple, and is therefore classifiable.

\begin{corollary}\label{class}
If $R_E$ is minimal, then the C$^*$-algebra $C_r^*(R_\xi)$ is classified by its Elliott invariant.
\end{corollary}

We now prove a noncommutative version of Proposition \ref{prod}.

\begin{prop}\label{bunce}
If $E=\xi^0(F)\cup\xi^1(F)$, then $C_r^*(R_E)$ is isomorphic to $C_r^*(R_F)\otimes M_{2^\infty}$, and $C_r^*(R_\xi)$ is isomorphic to $C_r^*(R_F)\otimes B$, where $B$ is the Bunce--Deddens algebra of type $2^\infty$.
\end{prop}

\begin{proof}
Let $S$ denote the tail-equivalence relation on $\{0,1\}^\bN$ with its usual \'etale topology, so that $C_r^*(S)\cong M_{2^\infty}$ (Example 9.2.6 of \cite{groupoid}). Let
\[\textstyle T=\{(w,z)\in \bT\times \bT\mid w=e^{2\pi i\theta}z\text{ for some }\theta\in\bZ[\frac12]\}\]
Give $T$ the topology with basic open sets $U_{W,\theta}=\{(z,e^{2\pi i\theta}z)\mid z\in W\}$, where $W\subseteq \bT$ is open and $\theta\in\bZ[\frac12]$. This makes $T$ a second-countable, locally compact, Hausdorff, \'etale groupoid, and $C_r^*(T)\cong B$, see 10.11.4 of \cite{blackadar}. It is not difficult to see that this topology is the same as the quotient topology received from $\tau:S\to T$, where $\tau=(\phi\times\phi)|_S$ and $\phi$ is as in Proposition \ref{prod}. %We show that this topology is the same as the quotient topology received from $\tau:S\to T$, where $\tau=(\phi\times\phi)|_S$ and $\phi$ is as in Proposition \ref{prod}. Suppose that $U\subseteq S$ is open and saturated; we must show that $\tau(U)$ is open. Select a pair $(w,z)$ in $\tau(U)$.

%Suppose first that $w$ and $z$ are not dyadic. Then $\tau^{-1}(w,z)=\{(x,y)\}$ for a pair $(x,y)$ in $S$, and $U$ is open, so there is a positive integer $k$ and finite strings $p$ and $q$ in $\{0,1\}^k$ such that $(x,y)\in\gamma(p,q)\subseteq U$. Set $V=\gamma(p,q)\cap(U_p\times U_q)$. Then
%\[(w,z)\in\tau(V)=U_{W,\theta}\subseteq\tau(U)\]
%where $W=\phi(U_p)$ and $\theta=\sum_{n=1}^\infty(y_n-x_n)2^{-n}=\sum_{n=1}^k(q_n-p_n)2^{-n}$.

%If $w$ and $z$ are dyadic, then $\tau^{-1}(w,z)=\{(x,y),(x',y')\}$ for two pairs $(x,y)$ and $(x',y')$ in $S$. This time find $\gamma(p,q)$ and $\gamma(p',q')$, such that $\tau^{-1}(w,z)\subseteq\gamma(p,q)\cup\gamma(p',q')\subseteq U$ and set $V=(\gamma(p,q)\cup\gamma(p',q'))\cap(V_{p,p'}\times V_{q,q'})$. Then
%\[(w,z)\in\tau(V)=U_{W,\theta}\subseteq\tau(U)\]
%where $W=\phi(V_{p,p'})$ and $\theta=\sum_{n=1}^\infty(y_n-x_n)2^{-n}$.

Define the map $\Psi:R_E\to R_F\times S$ by $\Psi(x,y)=((z_x,z_y),(j_x,j_y))$, using the notation in Proposition \ref{prod}. It is straightforward to check that this is an isomorphism of groupoids (which proves the first statement of the proposition). The composition of this isomorphism with $\id_{R_F}\times\tau$ is a quotient map whose fibres are identical to the fibres of $\pi$, which therefore factors to an isomorphism from $R_\xi$ to $R_F\times T$.
\end{proof}

%\begin{remark}
%Based on Lemma \ref{comp} and Proposition \ref{bunce}, it seems likely in general that $C_r^*(R_\xi)$ is isomorphic to an inductive limit of algebras of the form $A\otimes B$ where $A$ is AF and $B$ is the Bunce--Deddens algebra of type $2^\infty$. We do not investigate this further.
%\end{remark}

\section{Extensions of Cantor minimal systems}

We now turn to the second construction. We begin with the definition of an iterated function system.

\begin{definition}
An \emph{iterated function system} $(X,d_X,\F)$ is a complete metric space $(X,d_X)$ together with a finite set $\F$ of functions $f:X\to X$ such that there exists a real number $\lambda$ with $0<\lambda<1$ and $d_X(f(x),f(y))\leq\lambda d_X(x,y)$ for all points $x$ and $y$ in $X$ and all $f$ in $\F$.
\end{definition}

It is a standard fact that every iterated function system $(X,d_X,\F)$ possesses a unique non-empty compact subset $C$ of $X$ such that
\[\bigcup_{f\in\F}f(C)=C.\]
The set $C$ is called the \emph{attractor} of the iterated function system.

We now outline the construction of the extension in \cite{dps}. Let $(X,\phi)$ be a minimal dynamical system on the Cantor set, and let $(C,d_C,\F)$ be an iterated function system with attractor $C$, and assume that each $f$ in $\F$ is injective. We may assume that the Cantor minimal system is given by a Bratteli-Vershik system $(X_E,\phi_E)$ on a properly ordered Bratteli diagram, see Theorem 5.1 of \cite{cms}. By telescoping the diagram, we may assume that there are at least $\#\F+1$ edges between any two vertices in consecutive levels. To each edge $e$ in the diagram we assign a function $f_e$ in $\F\cup\{\id_C\}$ such that the following properties hold.
\begin{enumerate}[(i)]
    \item If $e$ is either maximal or minimal, then $f_e\neq\id_C$.
    \item For every $v$ in $V$, we have $\bigcup_{t(e)=v,f_e\neq\id_C}f_e(C)=C$.
    \item The set $\{e\mid f_e=\id_C\}$ contains an infinite path.
\end{enumerate}
We then define
\[X_n=\{(x,c)\in X_E\times C\mid c\in f_{x_1}\circ\cdots\circ f_{x_n}(C)\}\]
and $\tilde X=\bigcap_{n=1}^\infty X_n$. The factor map $\tilde\pi:\tilde X\to X_E$ (we reserve the letter $\pi$ for our groupoid homomorphism) is defined by $\tilde\pi(x,c)=x$. Notice that the space $\tilde X$ consists of two types of points, as follows.
\begin{enumerate}[(i)]
    \item If $(x,c)$ is in $\tilde X$ and $f_{x_n}\neq\id_C$ for infinitely many $n$, then the fibre $\tilde\pi^{-1}(\tilde\pi(x,c))$ consists of a single point $(x,c_x)$ due to the contractive factor $\lambda$ causing the diameters of the sets $f_{x_1}\circ\cdots\circ f_{x_n}(C)$ to shrink.
    \item If there is an $m$ with $f_{x_n}=\id_C$ for all $n\geq m$, then the fibre $\tilde\pi^{-1}(\tilde\pi(x,c))$ is equal to
    \[\{x\}\times f_{x_1}\circ\cdots\circ f_{x_m}(C)\]
    and is thus homeomorphic to $C$ since each function $f$ in $\F$ is injective.
\end{enumerate}
The self map $\tilde\phi$ of $\tilde X$ is then defined as follows, based on these two types of points.
\begin{enumerate}[(i)]
    \item If $x$ has infinitely many $n$ with $f_{x_n}\neq\id_C$, there is a unique point $c_x$ in $C$ with $\tilde\pi(x,c_x)=x$. In this case, $\phi_E(x)$ also has infinitely many such $n$, and we set $\tilde\phi(x,c)=(\phi_E(x),c_{\phi_E(x)})$.
    \item If there is an $m$ with $f_{x_n}=\id_C$ for all $n\geq m$, then
    \[\phi_E(x)=(y_1,y_2,\cdots,y_m,x_{m+1},x_{m+2},\ldots)\]
    for some edges $y_1,\ldots,y_m$ with $t(x_m)=t(y_m)$. In this case, we set
    \[\tilde\phi(x,c)=(\phi_E(x),f_{y_1}\circ\cdots\circ f_{y_m}\circ f_{x_m}^{-1}\circ\cdots\circ f_{x_1}^{-1}(c)).\]
\end{enumerate}

Before we define the factor map, we need to make a slight modification to the system $(\tilde X,\tilde\phi)$ so that Definition \ref{reg} will be satisfied. For $n\geq1$ we set
\[\F^{(n)}=\{f_1\circ f_2\circ\cdots\circ f_n\mid f_j\in\F\FOR1\leq j\leq n\}\]
adjust the edge assignment so that an edge $e$ in $E_n$ is assigned to a function $f_e$ in $\F^{(n)}\cup\{\id_C\}$ (note that this may require further telescoping of the diagram). The map $\tilde\phi:\tilde X\to\tilde X$ is then constructed in the same way.

\begin{definition}
    Define the map $\pi:R_{\tilde\phi}\to R_{\phi_E}$ by $\pi=\tilde\pi\times\tilde\pi|_{R_{\tilde\phi}}$.
\end{definition}

\begin{prop}
\begin{enumerate}[(i)]
    \item $R_{\tilde\phi}$ and $R_{\phi_E}$ are second-countable, locally compact, Hausdorff, \'etale groupoids.
    \item The map $\pi:R_{\tilde\phi}\to R_{\phi_E}$ is regular and satisfies the hypotheses of Proposition \ref{factor}.
\end{enumerate}
\end{prop}

\begin{proof}
Both groupoids arise from free actions of the integers, so (i) is immediate. The only nontrivial thing to prove for (ii) is that $\pi$ is regular. It suffices restrict our attention to $\tilde\pi$. By scaling the metric $d_C$, we may assume that $\diam C=1$. Fix $\ep>0$ and $x$ in $X_E$; we seek an open set $U$ of $X_E$ such that for all $y$ in $U$, either $d(\tilde\pi^{-1}(x),\tilde\pi^{-1}(y))<\ep$ or $\diam\tilde\pi^{-1}(y)<\ep$.

First suppose there are infinitely many $i$ with $f_{x_i}\neq\id_C$. Choose $k\geq1$ such that $\lambda^k<\ep$ and $f_{x_k}\neq\id_C$, and set $U=C(x_1,x_2,\ldots,x_k)$. If $y$ is in $U$ and $(y,c)$ and $(y,d)$ are in $\tilde\pi^{-1}(y)$, then $c$ and $d$ are both in $f_{y_1}\circ\cdots\circ f_{y_k}(C)=f_{x_1}\circ\cdots\circ f_{x_k}(C)$, and since $f_{x_k}$ has a contractive factor of at most $\lambda^k$ (this is where the modification is necessary), we have $d_C(c,d)\leq\lambda^k<\ep$ and hence $\diam\tilde\pi^{-1}(y)<\ep$.

Second, suppose that $f_{x_i}=\id_C$ eventually. Choose $k\geq1$ such that $2^{-k}<\ep$, $\lambda^k<\ep$, and $f_{x_i}=\id_C$ for all $i\geq k$, and set $U=C(x_1,x_2,\ldots,x_k)$. Let $y$ be in $U$; if $f_{y_i}\neq\id_C$ for some $i\geq k$, then $\diam f_{y_1}\circ\cdots\circ f_{y_i}(C)\leq\lambda^i\leq\lambda^k<\ep$, so $\diam\tilde\pi^{-1}(y)<\ep$ as in the first case. Otherwise, $f_{x_n}(C)=f_{y_n}(C)$ for all $n\geq1$, and $d_E(x,y)\leq2^{-N}<\ep$, which implies that $d(\tilde\pi^{-1}(x),\tilde\pi^{-1}(y))<\ep$.
\end{proof}

The factor map is also measure regular in this construction, but it will be easier to prove it later (Proposition \ref{meas}).

\begin{prop}
The dynamic asymptotic dimension of $R_{\tilde\phi}$ is $1$, and if the covering dimension of $C$ is finite, then $C_r^*(R_{\tilde\phi})$ has finite nuclear dimension.
\end{prop}

\begin{proof}
This is immediate from Theorem 1.3(i) of \cite{dad} and Theorem 8.6 of \cite{dad}.
\end{proof}

\begin{corollary}
If the covering dimension of $C$ is finite, then the C$^*$-algebra $C_r^*(R_{\tilde\phi})$ is classified by its Elliott invariant.
\end{corollary}

\section{$K$-theory}

Up until now we have proven parts (i) and (ii) of the second list in Theorem \ref{embedding}. In this section we turn to the remaining parts, as well as Theorem \ref{ifs}.

\subsection{Quotients of path spaces}

We begin by analyzing the subgroupoids $H\subseteq R_E$ and $H'\subseteq R_\xi$ introduced in Definition \ref{Hs}. Since the fibres of the map $\pi:R_E\to R_\xi$ consist of either one or two points, we have
\[H'=\{(x',y')\in R_\xi\mid\#\pi^{-1}(x',y')=2\}\]
and $H=\pi^{-1}(H')$. Clearly, $(x,y)$ is in $H$ if and only if the edges of $x$ and $y$ are eventually in $\xi^0(F)$ or eventually in $\xi^1(F)$. From part (ii) of Definition \ref{Hs}, we immediately obtain the following.

\begin{prop}\label{seq}
A sequence $(x^{(n)},y^{(n)})\to(x,y)$ in $H$ if and only if $x^{(n)}\to x$ and $y^{(n)}\to y$ in $X_E$ and, if $k\geq1$ is the least integer such that $x_i$ and $y_i$ are in $\xi^j(F)$ for $i\geq k$, then there is an integer $m\geq1$ such that $k$ is the least integer such that $x_i^{(n)}$ and $y_i^{(n)}$ are in $\xi^j(F)$ for all $i\geq k$ and $n\geq m$.
\end{prop}

Recall from part (iii) of Definition \ref{Hs} that $H_{2^n}$ denotes the elements of $H$ whose fibres have diameter greater than $2^{-n}$ in the metric $d_E^{(2)}$. The following fact will be convenient, the proof of which is straightforward.

\begin{lemma}\label{H2n} Let $(x,y)$ be in $H$. Then $(x,y)$ is in $H_{2^n}$ if and only if at least one of the following holds:
\begin{enumerate}[(i)]
    \item $x_k\in\xi^0(F)\FORAL k\geq n$
    \item $x_k\in\xi^1(F)\FORAL k\geq n$
    \item $y_k\in\xi^0(F)\FORAL k\geq n$
    \item $y_k\in\xi^1(F)\FORAL k\geq n$
\end{enumerate}
\end{lemma}

\begin{lemma}\label{H}
There is a $^*$-isomorphism $\beta:C_r^*(H')\oplus C_r^*(H')\to C_r^*(H)$ such that the diagram
\begin{center}
    % https://tikzcd.yichuanshen.de/?fbclid=IwAR0FybHEYl6QS_hMmNoKyCn60muRw7aO-eQAaFVrU3mcQkgAWpX_FaPCKwY#N4Igdg9gJgpgziAXAbVABwnAlgFyxMJZABgBpiBdUkANwEMAbAVxiRAGEB9AJwD0AqABQAJAOQBKEAF9S6TLnyEUAJnJVajFmy58hwyTLnY8BImWXr6zVohAAdOwGlpskBmOKiqi9StbbDo4OEHgAtvAABOyC7AbqMFAA5vBEoABm3BChSGQgOBBIqhrWbA4AxgSJLumZ2YgAjNT5SADMvpo29nYVYFWGIBlZOU0FDU10WAxsABYQEADWIO0ltgBGDqF0aHD5EevdWNxlDmhY1QO1hSOt45Mzc4vL-iB0G1s7EBGvdiFY4XD1aQUKRAA
\begin{tikzcd}
C_r^*(H') \arrow[d, equal] \arrow[r, "\gamma"]  & C_r^*(H')\oplus C_r^*(H') \arrow[d, "\beta"]  \\
C_r^*(H') \arrow[r, "\alpha"]                            &  C_r^*(H)              
\end{tikzcd}
\end{center}
is commutative, where $\alpha$ is as in Proposition \ref{factor}, and $\gamma(a)=(a,a)$.
\end{lemma}

\begin{proof}
Set
\[L_j=\{(x,y)\in H\mid \text{there exists }k\geq1\text{ such that }x_i,y_i\in\xi^j(F)\FORAL i\geq k\}\]
for $j=0,1$. Then $L_0\cup L_1$ is a partition of $H$ into two nonempty disjoint sets, and by Proposition \ref{seq}, they are clopen. $L_0$ and $L_1$ are therefore locally compact Hausdorff \'etale subgroupoids of $H$. Moreover, if $f_j$ is in $C_c(H)$ with support contained in $L_j$ for $j=0,1$, then for any $(x,y)$ in $H$,
\[(f_0\star f_1)(x,y)=\sum_{(x,z)\in R_E}f_0(x,z)f_1(z,y)=0\]
since all $(x,z)$ and $(z,y)$ lie in one and only one $L_j$. It follows that $C_c(H)=C_c(L_0)\oplus C_c(L_1)$. By partitioning the direct sum
\[\bigoplus_{x\in H^{(0)}}l^2(H_x)=\left(\bigoplus_{x\in L_0^{(0)}}l^2(H_x)\right)\oplus\left(\bigoplus_{x\in L_1^{(0)}}l^2(H_x)\right)\]
and representing $C_c(L_0)$ (respectively $C_c(L_1)$) on the left (respectively right) factor, it is apparent that $C_r^*(H)=C_r^*(L_0)\oplus C_r^*(L_1)$ as well. Define
\[\beta:C_c(H')\oplus C_c(H')\to C_c(L_0)\oplus C_c(L_1)\]
by $\beta(f,g)=(f\circ(\pi|_{L_0}),g\circ(\pi|_{L_1}))$. Similarly as in Proposition \ref{factor}, this is a $^*$-isomorphism which extends to the reduced algebras. It is straightforward to check that the diagram commutes on continuous compactly supported functions, and therefore on the completions.
\end{proof}

\begin{lemma}\label{H'}
$C_r^*(H')$ is Morita equivalent to $C_r^*(R_F)$, where $R_F$ is tail-equivalence on the diagram $(W,F)$. It follows that $K_*(C_r^*(H'))$ is isomorphic to $K_*(C_r^*(R_F))$.
\end{lemma}

\begin{proof}
Let $L_0$ be as in Lemma \ref{H}, and let
\[J=\{(x,y)\in L_0\mid x_n,y_n\in\xi^0(F)\FORAL n\geq1\},\]
endowed with the relative topology from $H$. The map $R_F\to J$ which sends the pair $(x,y)$ to the pair $(\xi^0(x),\xi^0(y))$ (applying $\xi^0$ to each edge of $x$ and $y$ in the obvious way) is clearly an isomorphism of groupoids. $J^{(0)}$ is thus compact and open (in $H$), hence $\upchi_{J^{(0)}}$ is a projection in $C_c(L_0)$. Then $\upchi_{J^{(0)}}C_r^*(L_0)$ is an equivalence bimodule between $C_r^*(L_0)\cong C_r^*(H')$ and $\upchi_{J^{(0)}}C_r^*(L_0)\upchi_{J^{(0)}}\cong C_r^*(J)\cong C_r^*(R_F)$. 
\end{proof}

\begin{lemma}\label{surj}
The map $\alpha_*:K_0(C_r^*(R_\xi))\to K_0(C_r^*(R_E))$ is surjective.
\end{lemma}

\begin{proof}
Since $C_r^*(R_E)$ is a unital AF-algebra, it suffices to show that, for a finite path $p$ in the diagram $(V,E)$, the characteristic function $\upchi_{\gamma(p,p)}$ is equivalent to a projection in $\alpha(C_c(R_\xi))$. If the last edge $p_n$ in $p$ is not in $\xi^0(F)\cup\xi^1(F)$, then $\upchi_{\gamma(p,p)}$ is in $\alpha(C_c(R_\xi))$, so there is nothing to do. Otherwise, let $k$ be the least integer in $\{1,2,\ldots,n\}$ such that $p_j$ is in $\xi^0(F)\cup\xi^1(F)$ for all $k\leq j\leq n$. For a string $\omega=(j_1,j_2,\ldots,j_{n-k+1})$ in $\{0,1\}^{n-k+1}$, let
\[q_{\omega}=(p_1,p_2,\ldots,p_{k-1},\xi^{j_1}(z_k),\xi^{j_2}(z_{k+1}),\ldots,\xi^{j_{n-k+1}}(z_n))\]
and let $L=\bigcup_{\omega,\eta\in\{0,1\}^{n-k+1}}\gamma(q_{\omega},q_{\eta})$. It is easy to see that $L$ is saturated. Let
\[v=\frac{1}{\sqrt{2^{n-k+1}}}\sum_{\omega\in\{0,1\}^{n-k+1}}\upchi_{\gamma(p,q_{\omega})}\]
which is a partial isometry such that $vv^*=\upchi_{\gamma(p,p)}$ and $v^*v=\frac{1}{2^{n-k+1}}\upchi_L$, the latter of which is in $\alpha(C_c(R_\xi))$.
\end{proof}

Readers with a background in quantum information theory may notice a similarity between the formula for $v$ in the above proof and the one for the $n$-fold tensor product of the Hadamard operator $\tilde H$ acting on a computational basis state (written in ket notation),
\[\tilde H^{\otimes n}\ket{\omega}_n=\frac{1}{\sqrt{2^n}}\sum_{\eta\in\{0,1\}^n}(-1)^{\omega\cdot \eta}\ket{\eta}_n\]
where $\omega\cdot\eta$ is the componentwise dot product computed mod 2. Indeed, if one identifies the functions $\upchi_{\gamma(p,q)}$ with matrix units as described in the preliminaries sections, the matrices $vv^*$ and $v^*v$ are respectively identified with the projections
\[\left[\begin{array}{ccccc}
    1 & 0 & 0 & \cdots & 0 \\
    0 & 0 & 0 & \cdots & 0 \\
    \vdots & \vdots & \vdots & \ddots & \vdots \\
    0 & 0 & 0 & \cdots & 0
\end{array}\right]\qquad\frac{1}{2^n}\left[\begin{array}{ccccc}
    1 & 1 & 1 & \cdots & 1 \\
    1 & 1 & 1 & \cdots & 1 \\
    \vdots & \vdots & \vdots & \ddots & \vdots \\
    1 & 1 & 1 & \cdots & 1
\end{array}\right]\]
in $M_{2^n}(\bC)$. A unitary equivalence between them is implemented by $\tilde H^{\otimes n}$.

Finally, we verify that the sets $\ol{H_n}^G$ have the $C$-extension property. In fact, we can find extensions that have the same norm.

\begin{prop}\label{cext}
    For every integer $n\geq1$, $H_n$ is closed in $R_E$, and for every $f$ in $C_c(H)$ with support contained in $H_n$, there exists $\tilde f$ in $C_c(R_E)$ such that $\tilde f|_{H_n}=f|_{H_n}$ and $\|\tilde f\|_r=\|f\|_r$.
\end{prop}

\begin{proof}
    For the first statement, it will suffice to show that $H_{2^n}$ is closed, for every integer $n\geq1$. If $(x,y)$ is not in $H_{2^n}$, then by Lemma \ref{H2n} there is an integer $k\geq n$ such that $t(x_k)=t(y_k)$ and such that the edges $x_n,x_{n+1},\ldots,x_k$ are not all in $\xi^0(F)$ (nor are they all in $\xi^1(F)$), and the edges $y_n,y_{n+1},\ldots,y_k$ are not all in $\xi^0(F)$ (nor are they all in $\xi^1(F)$). Setting $p=(x_1,x_2,\ldots,x_k)$ and $q=(y_1,y_2,\ldots,y_k)$, $\gamma(p,q)$ is an open set disjoint from $H_{2^n}$ that contains $(x,y)$.
    
    The support of $f$ is compact in $H$, and because the topology of $H$ is finer than that of $R_E$, it is compact in $R_E$ as well. Therefore, we may find an integer $k\geq1$ such that the support of $f$ is contained in $R_k$. By increasing $k$ if necessary, assume that $2^{-k}\leq\frac1n$, so that $H_n\subseteq H_{2^k}$. Define
    \[K=\{(x,y)\in R_k\mid x_i\in\xi^0(F)\FORAL i\geq k,\OR x_i\in\xi^1(F)\FORAL i\geq k\}.\]
    We claim the following.
    \begin{enumerate}[(i)]
        \item $K=R_k\cap H_{2^k}$, hence $K\subseteq H$.
        \item $K^u=R_k^u$ for all $u$ in $K^{(0)}$.
        \item $K$ is a compact subgroupoid of $R_k$.
        \item The relative topologies that $K$ receives from $R_E$ and $H$ are equal.
    \end{enumerate}
    (ii) is straightforward, as is (i) using Lemma \ref{H2n}. For (iii), it is clear that $K$ is a subgroupoid, and by (i) it is a closed subset of the compact set $R_k$, hence it is compact. Lastly, if $(x^{(n)},y^{(n)})$ is a sequence in $K$ converging to $(x,y)$ in $K$ in the topology of $R_E$, Proposition \ref{seq} implies that it also converges in $H$. This establishes (iv).
    
    The algebras $C_c(R_k)$ and $C_c(K)$ are complete (hence C$^*$-algebras) in the reduced norm since both groupoids are compact. The map
    \[C_c(R_k)\to C_c(K):g\mapsto g|_K\]
    is surjective by the Tietze extension theorem, and a $^*$-homomorphism by (ii) above. By (iv), $f|_K$ is continuous in the relative topology from $R_E$, so we may find a lift $\tilde f$ in $C_c(R_k)$ with the same norm. Since $R_k$ is compact and open in $R_E$, we may extend $\tilde f$ to $R_E$ by defining it to be zero on $R_E-R_k$. To check that $\tilde f|_{H_n}=f|_{H_n}$, note that they obviously coincide on $K$, but if $(x,y)$ is in $H_n-K$, then it is in $H_{2^k}-K$, and hence in $R_E-R_k$ by (i). Since $R_k$ contains the support of $f$, we have $f(x,y)=0$. But $\tilde f$ was extended to be zero on $R_E-R_k$, so $\tilde f(x,y)=0$ as well.
    \end{proof}

We are now ready to use the relative groups and the excision theorem to compute the $K$-theory $K_*(C_r^*(R_\xi))$. We have the six-term exact sequence

\begin{center}
    % https://tikzcd.yichuanshen.de/?fbclid=IwAR0FybHEYl6QS_hMmNoKyCn60muRw7aO-eQAaFVrU3mcQkgAWpX_FaPCKwY#N4Igdg9gJgpgziAXAbVABwnAlgFyxMJZAFgBoAGAXVJADcBDAGwFcYkQBpAfXIAoBhLgCcAegCpeAJS4AdGQA8sASiUgAvqXSZc+QigBMFanSat23PoNETpcxUoDcV8VK4BRFes0gM2PASJyIxoGFjZETi4ARgFhF2kPVQ0tP11A0n1jULMI7hjnG1kFZSTvXx0AlDJMkNNwyMs4wsSvFIq9ZEMakzDzaNjrVztlJybXFrVjGCgAc3giUAAzIQgAWyRDEBwIJCjkkGW13ZptpHJ9w-XEIK2dxGILlauyW6QAVkejxDeTu4BmT5XP6-DaTNRAA
\begin{tikzcd}
K_1(C_r^*(R_E)) \arrow[rr]   &  & K_0(C_r^*(R_\xi),C_r^*(R_E)) \arrow[rr] &  & K_0(C_r^*(R_\xi)) \arrow[dd] \\
                             &  &                                         &  &                              \\
K_1(C_r^*(R_\xi)) \arrow[uu] &  & K_1(C_r^*(R_\xi),C_r^*(R_E)) \arrow[ll] &  & K_0(C_r^*(R_E)) \arrow[ll]  
\end{tikzcd}
\end{center}
Using (in order) Theorem \ref{exc}, Lemma \ref{H}, Example 2.6 from \cite{haslehurst}, and Lemma \ref{H'}, we obtain
\begin{align}
    \nonumber K_j(C_r^*(R_\xi),C_r^*(R_E))&\cong K_j(C_r^*(H'),C_r^*(H))\\
    \nonumber&\cong K_j(C_r^*(H'),C_r^*(H')\oplus C_r^*(H'))\\
    \nonumber&\cong K_{1-j}(C_r^*(H'))\\
    \nonumber&\cong K_{1-j}(C_r^*(R_F))
\end{align}
The exact sequence may therefore be simplified to

\begin{center}
    % https://tikzcd.yichuanshen.de/?fbclid=IwAR0FybHEYl6QS_hMmNoKyCn60muRw7aO-eQAaFVrU3mcQkgAWpX_FaPCKwY#N4Igdg9gJgpgziAXAbVABwnAlgFyxMJZAFgBoAGAXVJADcBDAGwFcYkQBpAfXIAoBhLgCcAegCpeAJS4AdGQA8sASiUgAvqXSZc+QigBMFanSat23PoNETpcxUoDcV8VK4BRFes0gM2PASJyIxoGFjZETi4ARgFhF2kPVQ0tP11A0n1jULMI7hjnG1kFZSTvXx0AlDJMkNNwyMs4wsSvFIq9ZEMakzDzaNjrVztlJybXFrVjGCgAc3giUAAzIQgAWyRDEBwIJCjkkGW13ZptpHJ9w-XEIK2dxGILlauyW6QAVkejxDeTu4BmT5XP6-DaTNRAA
\begin{tikzcd}
0 \arrow[rr]   &  & 0 \arrow[rr] &  & K_0(C_r^*(R_\xi)) \arrow[dd] \\
                             &  &                                         &  &                              \\
K_1(C_r^*(R_\xi)) \arrow[uu] &  & K_0(C_r^*(R_F)) \arrow[ll] &  & K_0(C_r^*(R_E)) \arrow[ll]  
\end{tikzcd}
\end{center}

By Lemma \ref{surj}, the vertical map on the right is an order isomorphism. Exactness implies that the following map is zero, which in turn implies that the map after that is an isomorphism.

We have proven parts (iii) and (iv) of Theorem \ref{embedding}; it remains to prove part (v).
As these groupoids are principal, the tracial states $\tau$ on their C$^*$-algebras correspond exactly to invariant (that is, $\mu(r(U))=\mu(s(U))$ for any open bisection $U$) Borel measures $\mu$ on their unit spaces via
\[\tau(f)=\int_{G^{(0)}}f\,d\mu\]
for $f$ in $C_c(G)$. It therefore suffices to show that any invariant measure cannot see the non-one-to-one parts of $\pi$.

\begin{prop}
\begin{enumerate}[(i)]
    \item If $\mu$ is an invariant measure on $X_E$ then $\mu(H^{(0)})=0$.
    \item If $\mu$ is an invariant measure on $X_\xi$ then $\mu(H'^{(0)})=0$.
\end{enumerate}
\end{prop}

\begin{proof}
\begin{enumerate}[(i)]
    \item Every such measure $\mu$ is uniquely determined by a function $\nu:V\to[0,1]$ such that $\nu(v_0)=1$ and $\nu(v)=\sum_{e\in i^{-1}(v)}\nu(t(e))$, the correspondence being $\nu(v)=\mu(C(p))$, where $p$ is any finite path ending at $v$ (the invariance of $\mu$ makes the choice of $p$ immaterial). It suffices to show that the set $Y$ of all paths whose edges are eventually in $\xi^0(F)$ has measure zero; the proof is analogous for those paths eventually in $\xi^1(F)$, and $H^{(0)}$ is the union of these. Let
    \[Y_n=\{x\in X_E\mid x_m\in\xi^0(F)\FORAL m\geq n\}\]
    We have $Y=\bigcup_{n=1}^\infty Y_n$, so it further suffices to show that $\mu(Y_n)=0$ for all $n$. To this end, fix $n$ and, for $m\geq n+1$, let $A_m$ be the union of all cylinder sets of the form
    \[C(p_1,p_2,\ldots,p_n,q_{n+1},q_{n+2},\ldots,q_m)\]
    where $q_k$ is in $\xi^0(F)$ for $n+1\leq k\leq m$. For any cylinder set $C(p)$ in the union $A_m$, there are at least $2^{m-(n+1)}$ cylinder sets in total passing through the same vertices as the path $p$, due to the edges in $\xi^1(F)$ (and possibly more). Thus, by the sum given by $\nu$, we have $\mu(A_m)\leq2^{-m}$. Since $Y_n\subseteq A_m$ for every $m\geq n+1$, this completes the proof of (i).
    \item The set $H'^{(0)}$ is contained in the countable union $\bigcup_{p\in\P}\rho(C_\xi(p))$, so it suffices to show that $\mu(H'^{(0)}\cap\rho(C_\xi(p)))=0$ for every $p$ in $\P$. Using part (ii) of Lemma \ref{comp}, we may identify $\rho(C_\xi(p))$ with $X_F^{(w)}\times \bT$ for some $w$ in $W$. Let $n\geq1$ and let $\gamma$ be any open arc on $\bT$ of normalized Lebesgue length $2^{-n}$. By invariance of $\mu$, the disjoint sets $X_F^{(w)}\times e^{2\pi ik/2^n}\gamma$ for $k=0,1,\ldots,2^n-1$ are each contained in an open set in $X_\xi$ of equal measure, so we have
    \[\mu(X_F^{(w)}\times\gamma)\leq2^{-n}\mu(X_F^{(w)}\times \bT)\]
    Thus we obtain that $\mu(X_F^{(w)}\times\{z\})=0$ for any $z$ in $\bT$. As $H'^{(0)}\cap\rho(C_\xi(p))$ is identified with $X_F^{(w)}\times\exp(2\pi i\bZ[\frac12])$ and $\exp(2\pi i\bZ[\frac12])$ is countable, we are done.\qedhere
\end{enumerate}
\end{proof}
%The order isomorphism between $K_0(C_r^*(R_\xi))$ and $K_0(C_r^*(R_E))$ is quite clean, but the isomorphism between $K_1(C_r^*(R_\xi))$ and $K_0(C_r^*(R_F))$ is buried throughout several results. We elucidate the picture of $K_1(C_r^*(R_\xi))$ through some examples.

\begin{example}
Consider the situation in Example \ref{simp}. By Theorem \ref{embedding}, we have $K_0(C_r^*(R_\xi))\cong\bZ[\frac12]$ with its usual order, and $K_1(C_r^*(R_\xi))\cong\bZ$ with generator given by the class of the identity function $z$ in $C(\bT)=C(R_\xi^{(0)})\subseteq C_r^*(R_\xi)$.
\end{example}

\begin{example}
Consider the situation in Example \ref{cantor}. By Theorem \ref{embedding}, we have $K_0(C_r^*(R_\xi))\cong K_1(C_r^*(R_\xi))\cong\bZ[\frac12]$ with the usual order in the $K_0$ case. The $K_1$-group is generated by the partial unitaries $\upchi_U\otimes z$ in $C(X_F)\otimes C(\bT)=C(R_\xi^{(0)})$, where $U$ is a clopen subset of $X_F$ (see Proposition 3.4 of \cite{haslehurst} and the discussion preceding it).
\end{example}

%\begin{example}
%Consider the situation in Example \ref{2}. By Theorem \ref{embedding}, we have $K_0(C_r^*(R_\xi))\cong$ with its usual order and $K_1(C_r^*(R_\xi))\cong\bZ$.
%\end{example}

%\begin{example}
%Consider the situation in Example \ref{3}. By Theorem \ref{embedding}, we have $K_0(C_r^*(R_\xi))\cong\bZ[\frac13]$ with its usual order and $K_1(C_r^*(R_\xi))\cong\bZ$. We show that the generator of $K_1(C_r^*(R_\xi))$ is given by
%\end{example}

\subsection{Extensions of Cantor minimal systems}

We now prove Theorem \ref{ifs}. Until we are ready to employ Theorem \ref{exc} (up until Proposition \ref{cext2}), it will be convenient to write the groupoids using the products $\tilde X\times\bZ$ and $X_E\times\bZ$ instead of $R_{\tilde\phi}$ and $R_{\phi_E}$, respectively. With this notation, the factor map $\pi:\tilde X\times\bZ\to X_E\times\bZ$ is $\tilde\pi\times\id_\bZ$.

In this situation, the groupoids $H'$ and $H$ from Definition \ref{Hs} take the form
\[H'=\{(x,l)\in X_E\times\bZ\mid \text{there exists }k\geq1\text{ such that }f_{x_i}=\id_C\FORAL i\geq k\}\]
and
\[H=\{(x,c,l)\in\tilde X\times\bZ\mid\text{there exists }n\geq1\text{ such that }f_{x_i}=\id_C\FORAL i\geq k\}.\]
We give a description of the topologies on $H'$ and $H$ that arise from the metrics in Definition \ref{Hs}.

\begin{prop}\label{seq2}
\begin{enumerate}[(i)]
    \item $(x^{(n)},l_n)\to(x,l)$ in $H'$ if and only if $l_n=l$ eventually, $x^{(n)}\to x$ in $X_E$, and if $k\geq1$ is the least integer such that \emph{$f_{x_i}=\id_C$} for all $i\geq k$, then there exists an integer $m\geq1$ such that $k$ is the least integer such that \emph{$f_{x_i^{(n)}}=\id_C$} for all $i\geq k$ and $n\geq m$.
    \item $(x^{(n)},c_n,l_n)\to(x,c,l)$ in $H$ if and only if $(x^{(n)},l_n)\to(x,l)$ in $H'$ and $c_n\to c$ in $C$.
\end{enumerate}
\end{prop}

The next result allows us to identify $H$ with $H'\times C$. The proof is routine.

\begin{lemma}\label{REC}
\begin{enumerate}[(i)]
    \item The map $\Psi:H\to H'\times C$ defined by
    \[\Psi(x,c,l)=((x,l),f^{-1}_{x_n}\circ\cdots\circ f^{-1}_{x_1}(c))\]
    is well-defined as long as \emph{$f_{x_i}=\id_C$} for $i\geq n$, and is an isomorphism of groupoids ($C$ is regarded as the cotrivial groupoid $\{(c,c)\mid c\in C\}$).
    \item If $\pi^C:H'\times C\to H'$ is defined by $\pi^C((x,l),c)=(x,l)$, then $\pi^C\circ\Psi=\pi$.
\end{enumerate}
\end{lemma}

\begin{lemma}
$C_r^*(H')$ is Morita equivalent to an AF-algebra.
\end{lemma}

\begin{proof}
The proof is similar to that of Lemma \ref{H'}. Consider the subdiagram $(W,F)$ of $(V,E)$ formed by all infinite paths $x$ such that $f_{x_n}=\id_C$ for all $n\geq1$, and let $R_F$ be the AF-groupoid associated to $(W,F)$. Note that there is at least one such path by property (iii) of the edge assignment. Let
\[J=\{(x,l)\in H'\mid f_{x_n}=f_{y_n}=\id_C\FORAL n\geq1,\text{ where }y=\phi_E^l(x)\}.\]
It is straightforward to verify that $J$ is an open subgroupoid of $H'$. The map $J\to R_F$ defined by $(x,l)\mapsto(x,\phi_E^l(x))$ is an isomorphism of groupoids. $J^{(0)}$ is thus compact and open (in $H'$), hence $\upchi_{J^{(0)}}$ is a projection in $C_c(H')$. Then $\upchi_{J^{(0)}}C_r^*(H')$ is an equivalence bimodule between $C_r^*(H')$ and $\upchi_{J^{(0)}}C_r^*(H')\upchi_{J^{(0)}}\cong C_r^*(J)\cong C_r^*(R_F)$.
\end{proof}

\begin{prop}\label{meas}
The map $\pi:\tilde X\times\bZ\to X_E\times\bZ$ is measure regular.
\end{prop}

\begin{proof}
By Lemma \ref{REC}, we may replace $\pi$ with $\pi^C$. Fix a point $c_0$ in $C$ and define the map $\mu:H'\to H'\times C$ by $\mu(x,l)=((x,l),c_0)$. Clearly $\mu$ is a continuous groupoid homomorphism and $\pi^C\circ\mu=\id_{H'}$, so Proposition 7.15 from \cite{exc} applies.
\end{proof}

\begin{prop}\label{cext2}
For every integer $n\geq1$ and every $f$ in $C_c(H)$ with support contained in $H_n$, there exists $\tilde f$ in $C_c(\tilde X\times\bZ)$ such that $\tilde f|_{H_n}=f|_{H_n}$ and $\|\tilde f\|_r=\|f\|_r$.
\end{prop}

\begin{proof}
If $(x,l)$ is in $H'$, then neither $(x,x^{\max})$ nor $(x,x^{\min})$ are in $R_E$ by property (i) of the edge assignment $e\mapsto f_e$ in section 4, so $(x,\phi_E^l(x))$ is in $R_E$ by the observation at the end of section 2. Thus $H'$ is contained in the union of the compact, open subgroupoids $R_k$ for $k\geq1$, and since $\pi$ is a continuous and proper groupoid homomorphism, $H$ is contained in the union of the compact, open subgroupoids $\pi^{-1}(R_k)$ of $\tilde X\times\bZ$.

As in Proposition \ref{cext}, we may choose an integer $k\geq1$ such that the support of $f$ is contained in $\pi^{-1}(R_k)$. By increasing $k$ if necessary, assume that $\lambda^k\leq\frac1n$ (recall that $\lambda$ is the contracting factor for the iterated function system). Define
\[L=\{(x,c,l)\in H\mid f_{x_i}=\id_C\FORAL i\geq k\}.\]
We claim that $H_n\subseteq L$. Indeed, if $(x,c,l)$ is not in $L$, then there exists an integer $i\geq k$ such that $f_{x_i}$ is in $\F^{(i)}$. Then $\diam\pi^{-1}(x,l)\leq\lambda^i\leq\lambda^k\leq\frac1n$, so $(x,c,l)$ is not in $H_n$.
Define
\[K=\pi^{-1}(R_k)\cap L.\]
We are abusing notation slightly, since $R_k$ is a subset of $R_E$ and $L$ is a subset of $\tilde X\times\bZ$. The important point is that $K$ consists of those elements $(x,c,l)$ in $L$ such that the pair $(x,\phi_E^l(x))$ is in $R_k$.

We claim that
\begin{enumerate}[(i)]
    \item $K^u=\pi^{-1}(R_k)^u$ for all $u$ in $K^{(0)}$.
    \item $K$ is a compact subgroupoid of $\pi^{-1}(R_k)$.
    \item The relative topologies that $K$ receives from $\tilde X\times\bZ$ and $H$ are equal.
\end{enumerate}
The range map is simply $r(x,c,l)=(x,c,0)$, so (i) is obvious. $K$ is clearly closed under multiplication, and since $(x,\phi_E^l(x))$ is in $R_k$, we have $f_{\phi_E(x)_i}=f_{x_i}=\id_C$ for all $i\geq k$. Then, since $(x,c,l)^{-1}=(\tilde\phi^l(x,c),-l)$ and $\tilde\pi(\tilde\phi^l(x,c))=\phi_E^l(x)$, we see that $K$ is closed under inverses as well. Next, $L$ is closed in $\tilde X\times\bZ$, for if $(x,c,l)$ is not in $L$, choose an integer $i\geq k$ such that $f_{x_i}\neq\id_C$, let $p=(x_1,x_2,\ldots,x_i)$, and then
\[[(C(p)\times C)\cap\tilde X]\times\bZ\]
is an open set disjoint from $L$ that contains $(x,c,l)$. It follows that $K$ is compact. Lastly, if $(x^{(n)},c_n,l_n)$ is a sequence in $K$ converging to $(x,c,l)$ in $K$ in the topology of $\tilde X\times\bZ$, Proposition \ref{seq2} implies that it also converges in $H$. This establishes (iii).

The rest of the proof proceeds analogously to that of Proposition \ref{cext}, with $R_k$ replaced with $\pi^{-1}(R_k)$.
\end{proof}

We may now employ Theorem \ref{exc} and begin computing the $K$-theory. We begin with a description of the relative groups.

\begin{lemma}\label{kunneth}
We have
\[K_0(C_r^*(R_{\phi_E}),C_r^*(R_{\tilde\phi}))\cong K_0(C_r^*(H'))\otimes K^{-1}(C)\]
and
\[K_1(C_r^*(R_{\phi_E}),C_r^*(R_{\tilde\phi}))\cong K_0(C_r^*(H'))\otimes(K^0(C)/\bZ)\]
\end{lemma}

\begin{proof}
By Theorem \ref{exc} and Lemma \ref{REC}, we have
\[K_*(C_r^*(R_{\phi_E}),C_r^*(R_{\tilde\phi}))\cong K_*(C_r^*(H'),C_r^*(H))\cong K_*(C_r^*(H'),C_r^*(H'\times C))\]
By combining the obvious isomorphism $C_r^*(H'\times C)\cong C_r^*(H')\otimes C(C)$ with the K\"unneth Theorem for tensor products \cite{sch}, we have a natural isomorphism
\[K_*(C_r^*(H'\times C))\cong K_0(C_r^*(H'))\otimes K^*(C)\]
under which the $K_0$ map induced by the inclusion $C_r^*(H')\subseteq C_r^*(H'\times C)$ becomes $K_0(C_r^*(H'))\to K_0(C_R^*(H'))\otimes K^0(C):g\mapsto g\otimes1$. We have the six-term exact sequence

\begin{center}
    % https://tikzcd.yichuanshen.de/?fbclid=IwAR0FybHEYl6QS_hMmNoKyCn60muRw7aO-eQAaFVrU3mcQkgAWpX_FaPCKwY#N4Igdg9gJgpgziAXAbVABwnAlgFyxMJZAFgBoAGAXVJADcBDAGwFcYkQBpAfXIAoBhLgCcAegCpeAJS4AdGQA8sASiUgAvqXSZc+QigBMFanSat23PoNETpcxUoDcV8VK4BRFes0gM2PASJyIxoGFjZETi4ARgFhF2kPVQ0tP11A0n1jULMI7hjnG1kFZSTvXx0AlDJMkNNwyMs4wsSvFIq9ZEMakzDzaNjrVztlJybXFrVjGCgAc3giUAAzIQgAWyRDEBwIJCjkkGW13ZptpHJ9w-XEIK2dxGILlauyW6QAVkejxDeTu4BmT5XP6-DaTNRAA
\begin{tikzcd}
K_0(C_r^*(H'))\otimes K^{-1}(C) \arrow[rr]   &  & K_0(C_r^*(H'),C_r^*(H'\times C)) \arrow[rr] &  & K_0(C_r^*(H')) \arrow[dd] \\
                             &  &                                         &  &                              \\
K_1(C_r^*(H')) \arrow[uu] &  & K_1(C_r^*(H'),C_r^*(H'\times C)) \arrow[ll] &  & K_0(C_r^*(H'))\otimes K^0(C) \arrow[ll]  
\end{tikzcd}
\end{center}
As $C$ is compact, there is a short exact sequence
\begin{center}
    % https://tikzcd.yichuanshen.de/?fbclid=IwAR0FybHEYl6QS_hMmNoKyCn60muRw7aO-eQAaFVrU3mcQkgAWpX_FaPCKwY#N4Igdg9gJgpgziAXAbVABwnAlgFyxMJZABgBpiBdUkANwEMAbAVxiRGJAF9T1Nd9CKAEzkqtRizYAdKQCMAWlx4gM2PASIAWUdXrNWiEAGkAesQAUAYQCUS3moFEAbDvH62pizYD0MhXZU+dUFkAA5XPUlDDk4xGCgAc3giUAAzACcIAFskMhAcCCQARm40zJzEIuoCpCFSkAzs2urCxABmesaKtpakTVjOIA
\begin{tikzcd}
0 \arrow[r] & \bZ \arrow[r] & K^0(C) \arrow[r] & K^0(C)/\bZ \arrow[r] & 0
\end{tikzcd}
\end{center}
and since $C_r^*(H')$ is Morita equivalent to an AF-algebra, $K_0(C_r^*(H'))$ is a dimension group, hence torsion free. Thus we may tensor the sequence above with $K_0(C_r^*(H'))$ and obtain that
\begin{center}
    % https://tikzcd.yichuanshen.de/?fbclid=IwAR0FybHEYl6QS_hMmNoKyCn60muRw7aO-eQAaFVrU3mcQkgAWpX_FaPCKwY#N4Igdg9gJgpgziAXAbVABwnAlgFyxMJZABgBpiBdUkANwEMAbAVxiRGJAF9T1Nd9CKAEzkqtRizYAdKQCMAWlx4gM2PASIAWUdXrNWiEAGkAesQAUAYQCUS3moFEAbDvH62pizYD0MhXZU+dUFkAA5XPUlDDk4xGCgAc3giUAAzACcIAFskMhAcCCQARm40zJzEIuoCpCFSkAzs2urCxABmesaKtpakTVjOIA
\begin{tikzcd}
0 \arrow[r] & K_0(C_r^*(H')) \arrow[r] & K_0(C_r^*(H'))\otimes K^0(C) \arrow[r] & K_0(C_r^*(H'))\otimes(K^0(C)/\bZ) \arrow[r] & 0
\end{tikzcd}
\end{center}
is exact. In particular, the map $K_0(C_r^*(H'))\to K_0(C_r^*(H'))\otimes K^0(C)$ is injective, so the right vertical arrow in the six-term sequence is injective. Combining the fact that $K_1(C_r^*(H'))=0$ with exactness gives the conclusion.
\end{proof}

\begin{lemma}\label{pims}
The map $\alpha_*:K_j(C_r^*(R_{\phi_E}))\to K_j(C_r^*(R_{\tilde\phi}))$ is injective for $j=0,1$.
\end{lemma}

\begin{proof}
Consider the commutative diagram
\begin{center}
    % https://tikzcd.yichuanshen.de/?fbclid=IwAR0FybHEYl6QS_hMmNoKyCn60muRw7aO-eQAaFVrU3mcQkgAWpX_FaPCKwY#N4Igdg9gJgpgziAXAbVABwnAlgFyxMJZABgBoBGAXVJADcBDAGwFcYkQBpAfXIAoBhLgCcAegCpeAcQDkASlkgAvqXSZc+QijIAmanSat23PoNETJ8pSpAZseAkXIU9DFm0ScRxXgA0uAUQVlVTsNIm1nGldDDw4vXwCg61t1BxQAZkj9NyMub1NxKTkkkNTNZEzdKIN3TjyBYUKLEps1e3KnKuyYz28AHT68RlgAAh8WlPbw0i7o2rj+waxhmDGJtrCUJ3SXGvZ+XgGh0Z9SAYAjAC110LTkCJ3qnI8Do+WTs76rm7LHUmJds8QAc-P5Pt8rKUpigIgCnj0QQFwddIa1buUACxZObsYhKPQwKAAc3gRFAADMhBAALZIMggHAQJDaYIgSk05k0RlITLdWoDLBQAC0AzQAAssAExKj2bTELzuYgsXz2AKIDh6FxpazZUhlYryCzrLrEE4GUzEAA2HVUuWWrkWgDs8P5fUFIqWK1FEq1IBojCwYFqcAgAagfpAYpg9HDiDAzEYjC59GW7EgQZltqQzvNSAArC7VW71Zrtcas0qHfnCx4BkxxaWI3AJeScEhyDaOfKq4gczja3162LG525RFc1aayA64wG77R3Se2b+9PB7Ph-Py13xwb6SuBlAsLSF92J+Rl3sB4fj1u5eR6QaO7ekAAOHsAThP9onb5VV6PmZdjmirvlOB4ASe54fjQ5wwGAsaOsq+5rnO2qUIoQA
\begin{tikzcd}
                                               & {C(X_E,\bZ)} \arrow[r, "\id-\phi_E"]                                                    & {C(X_E,\bZ)} \arrow[ddd, "\alpha", bend left=83]                    &                                                             &   \\
K_1(C_r^*(R_{\phi_E})) \arrow[r, "\delta_1"] \arrow[d, "\alpha_*"] & K^0(X_E) \arrow[r, "\id-\phi_E*"] \arrow[d, "\tilde\pi^*"] \arrow[u, "\dim"] & K^0(X_E) \arrow[r, "\iota_*"] \arrow[d, "\tilde\pi^*"] \arrow[u, "\dim"] & K_0(C_r^*(R_{\phi_E})) \arrow[r] \arrow[d, "\alpha_*", shift right] & 0 \\
K_1(C_r^*(R_{\tilde\phi})) \arrow[r, "\delta_1"]                        & K^0(\tilde X) \arrow[r, "\id-\tilde\phi_*"] \arrow[d, "\dim"]             & K^0(\tilde X) \arrow[r, "\iota_*"] \arrow[d, "\dim"]                  & K_0(C_r^*(R_{\tilde\phi}))                                               &   \\
                                               & {C(\tilde X,\bZ)} \arrow[r, "\id-\tilde\phi"]                                               & {C(\tilde X,\bZ)}                                                     &                                                             &  
\end{tikzcd}
\end{center}
The second and third rows are extracted from the Pimsner-Voiculescu exact sequence \cite{pv}, therefore they are exact. By the map $\alpha:C(X_E,\bZ)\to C(\tilde X,\bZ)$, we mean $\alpha(f)=f\circ\tilde\pi$, which is clearly an injective group homomorphism.

We consider the $j=0$ case first. Let $x$ be in $K_0(C_r^*(R_{\phi_E}))$ such that $\alpha_*(x)=0$. Since the top $\iota_*$ is surjective (by exactness), there is a $y$ in $K^0(X_E)$ with $\iota_*(y)=x$. We have $\iota_*(\tilde\pi^*(y))=\alpha_*(\iota_*(y))=\alpha_*(x)=0$, so by exactness again, there is some $z$ in $K^0(\tilde X)$ such that $(\id-\tilde\phi_*)(z)=\tilde\pi^*(y)$. Denote $f=\dim(y)$ and $g=\dim(z)$. Commutativity of the diagram implies that $f\circ\tilde\pi=g-g\circ\tilde\phi^{-1}$ in $C(\tilde X,\bZ)$, which means that $g-g\circ\tilde\phi^{-1}$ is constant on the fibres of $\tilde\pi$. Thus, if $w$ and $w'$ are points in $\tilde X$ with $\tilde\pi(w)=\tilde\pi(w')$, we have
\[g(w)-g(w')=g(\tilde\phi^{-1}(w))-g(\tilde\phi^{-1}(w'))\]
The above equation implies that the difference $g(w)-g(w')$ remains constant as $w$ and $w'$ run through their orbits under $\tilde\phi$ in tandem. Since $\tilde\phi$ is minimal and there are fibres that consist of a single point, we must have $g(w)-g(w')=0$, that is, $g$ must also be constant on the fibres of $\tilde\pi$. Therefore there is a $g_0$ in $C(X_E,\bZ)$ with $g=g_0\circ\tilde\pi$. Then
\[f\circ\tilde\pi=g-g\circ\tilde\phi^{-1}=g_0\circ\tilde\pi-g_0\circ\tilde\pi\circ\tilde\phi^{-1}=(g_0-g_0\circ\phi_E^{-1})\circ\tilde\pi\]
Since $\alpha:C(X_E,\bZ)\to C(\tilde X,\bZ)$ is injective, we have $f=g_0-g_0\circ\phi_E^{-1}=(\id-\phi_E)(g_0)$, and since $\dim:K^0(X_E)\to C(X_E,\bZ)$ is an isomorphism ($X_E$ is totally disconnected), there is one and only one $y'$ in $K^0(X_E)$ with $\dim(y')=g_0$. By commutativity we have
\[(\id-\phi_E*)(y')=\dim^{-1}\circ(\id-\phi_E)\circ\dim(y')=\dim^{-1}\circ(\id-\phi_E)(g_0)=\dim^{-1}(f)=y\]
Finally,
\[x=\iota_*(y)=\iota_*((\id-\phi_E*)(y'))=0\]
by exactness.

The $j=1$ case is simpler, since $K_1(C_r^*(R_{\phi_E}))\cong\bZ$ is generated by the class of the generating unitary $u$ in $C_c(R_{\phi_E})$, so we need only show that $\alpha_*([u])$ is nonzero. We have
\[\tilde\pi^*(\delta_1([u]))=\tilde\pi^*(-[1_{C(X_E)}])=-[1_{C(\tilde X)}]\]
and both $[1_{C(X_E)}]$ and $[1_{C(\tilde X)}]$ are nonzero since both spaces are compact. Thus
\[\delta_1(\alpha_*([u]))=\tilde\pi^*(\delta_1([u]))\neq0\]
hence $\alpha_*([u])$ cannot be zero.
\end{proof}

Returning to the six-term exact sequence with the relative groups $K_*(C_r^*(R_{\phi_E}),C_r^*(R_{\tilde\phi}))$, 

\begin{center}
    % https://tikzcd.yichuanshen.de/?fbclid=IwAR0FybHEYl6QS_hMmNoKyCn60muRw7aO-eQAaFVrU3mcQkgAWpX_FaPCKwY#N4Igdg9gJgpgziAXAbVABwnAlgFyxMJZAFgBoAGAXVJADcBDAGwFcYkQBpAfXIAoBhLgCcAegCpeAJS4AdGQA8sASiUgAvqXSZc+QigBMFanSat23PoNETpcxUoDcV8VK4BRFes0gM2PASJyIxoGFjZETi4ARgFhF2kPVQ0tP11A0n1jULMI7hjnG1kFZSTvXx0AlDJMkNNwyMs4wsSvFIq9ZEMakzDzaNjrVztlJybXFrVjGCgAc3giUAAzIQgAWyRDEBwIJCjkkGW13ZptpHJ9w-XEIK2dxGILlauyW6QAVkejxDeTu4BmT5XP6-DaTNRAA
\begin{tikzcd}
K_1(C_r^*(R_{\tilde\phi})) \arrow[rr]   &  & K_0(C_r^*(R_{\phi_E}),C_r^*(R_{\tilde\phi})) \arrow[rr] &  & K_0(C_r^*(R_{\phi_E})) \arrow[dd] \\
                             &  &                                         &  &                              \\
K_1(C_r^*(R_{\phi_E})) \arrow[uu] &  & K_1(C_r^*(R_{\phi_E}),C_r^*(R_{\tilde\phi})) \arrow[ll] &  & K_0(C_r^*(R_{\tilde\phi})) \arrow[ll]  
\end{tikzcd}
\end{center}
Lemma \ref{kunneth} and the fact that $K_1(C_r^*(R_{\phi_E}))\cong\bZ$ simplify this to

\begin{center}
    % https://tikzcd.yichuanshen.de/?fbclid=IwAR0FybHEYl6QS_hMmNoKyCn60muRw7aO-eQAaFVrU3mcQkgAWpX_FaPCKwY#N4Igdg9gJgpgziAXAbVABwnAlgFyxMJZAFgBoAGAXVJADcBDAGwFcYkQBpAfXIAoBhLgCcAegCpeAJS4AdGQA8sASiUgAvqXSZc+QigBMFanSat23PoNETpcxUoDcV8VK4BRFes0gM2PASJyIxoGFjZETi4ARgFhF2kPVQ0tP11A0n1jULMI7hjnG1kFZSTvXx0AlDJMkNNwyMs4wsSvFIq9ZEMakzDzaNjrVztlJybXFrVjGCgAc3giUAAzIQgAWyRDEBwIJCjkkGW13ZptpHJ9w-XEIK2dxGILlauyW6QAVkejxDeTu4BmT5XP6-DaTNRAA
\begin{tikzcd}
K_1(C_r^*(R_{\tilde\phi})) \arrow[rr]   &  & K_0(C_r^*(H'))\otimes K^{-1}(C) \arrow[rr] &  & K_0(C_r^*(R_{\phi_E})) \arrow[dd] \\
                             &  &                                         &  &                              \\
\bZ \arrow[uu] &  & K_0(C_r^*(H'))\otimes(K^0(C)/\bZ) \arrow[ll] &  & K_0(C_r^*(R_{\tilde\phi})) \arrow[ll]
\end{tikzcd}
\end{center}
By Lemma \ref{pims}, the top right and bottom left maps are zero, giving the conclusion of Theorem \ref{ifs}.

\begin{example}
Let $m\geq1$ be an integer, $C=[0,1]^m$, and define the functions $f_\delta:C\to C$ by $f_\delta(x)=\frac12(x+\delta)$ for $\delta$ in $\{0,1\}^m$. Let $(V,E)$ be the Bratteli diagram with $\#V_n=1$ for all $n$, and $\#E_n=2^{mn}+1$ for all $n$. Any order on the edges will do, and, for a fixed $n\geq1$, assign each function in $\F^{(n)}\cup\{\id_C\}$ to an edge in $E_n$, making sure that $\id_C$ is assigned neither the maximal edge nor the minimal edge. Since $C$ is contractible, part (i) of Corollary \ref{ifs2} implies that the inclusion map $\alpha:C_r^*(R_{\phi_E})\to C_r^*(R_{\tilde\phi})$ induced by the factor map induces isomorphisms $\alpha_*:K_0(C_r^*(R_{\phi_E}))\to K_0(C_r^*(R_{\tilde\phi}))$ and $\alpha_*:\bZ\to K_1(C_r^*(R_{\tilde\phi}))$.
\end{example}

\begin{example}
Let $C=\{0,1\}^\bN$ with the metric $d_C(x,y)=\inf\{1,2^{-k}\mid x_j=y_j\FORAL1\leq j\leq k\}$. Define the functions $f_j:C\to C$ by
\[f_j(\{j_1,j_2,j_3,\ldots\})=\{j,j_1,j_2,j_3,\ldots\}\]
for $j=0,1$. Let $(V,E)$ be the Bratteli diagram with $\#V_n=1$ for all $n$ and $\#E_n=2^n+1$ for all $n$. Any order on the edges will do, and, for a fixed $n\geq1$, assign each function in $\F^{(n)}\cup\{\id_C\}$ to an edge in $E_n$, making sure that $\id_C$ is assigned neither the maximal edge nor the minimal edge. We have $K^{-1}(C)=0$, so by part (ii) of Corollary \ref{ifs2}, the inclusion map $\alpha:C_r^*(R_{\phi_E})\to C_r^*(R_{\tilde\phi})$ induces an isomorphism $\alpha_*:\bZ\to K_1(C_r^*(R_{\tilde\phi}))$ and 
\[K_0(C_r^*(R_{\tilde\phi}))/K_0(C_r^*(R_{\phi_E}))\cong C(C,\bZ)/\bZ\]
where we identify $\bZ$ with the subgroup of $C(C,\bZ)$ of constant functions.
\end{example}

More generally, if the iterated function system $(C,d_C,\F)$ satisfies the \emph{strong separation condition}, that is, $f(C)\cap f'(C)=\mt$ whenever $f\neq f'$ in $\F$, then $(C,d_C,\F)$ is necessarily topologically conjugate to the system in the previous example (with $\{0,1\}^\bN$ replaced with $\{0,1,2,\ldots,m\}^\bN$, where $m=\#\F-1$).

\begin{example}\label{sier}
Let $\triangle$ denote the solid equilateral triangle in $\bR^2$ with corners at $(0,0)$, $(1,0)$, and $(\frac12,\frac{\sqrt3}{2})$. The Sierpi\'nski triangle $C$ is obtained by successively removing a countable number of disjoint open triangles from $\triangle$. More specifically, for $j=0,1,2$, define the three functions $f_j:C\to C$ by
\[f_0(x)=\frac12x\qquad f_1(x)=\frac12x+\begin{bmatrix}
           1/2 \\
           0
           \end{bmatrix}\qquad f_2(x)=\frac12x+\begin{bmatrix}
           1/4 \\
           \sqrt3/4
           \end{bmatrix}\]
Then $C$ is the attractor of the system $(\triangle,\{f_0,f_1,f_2\})$. Let $(V,E)$ be the Bratteli diagram with $\#V_n=1$ for all $n$ and $\#E_n=3^n+1$ for all $n$. As before, any order on the edges will do, and assign the functions similarly. We verify that $K^0(C)\cong\bZ$ and $K^{-1}(C)\cong\bZ^\infty$, where $\bZ^\infty$ denotes the group of all sequences of integers that are eventually zero.

Let $U$ denote the union of open triangles removed from the solid equilateral triangle (each open triangle is homeomorphic to $\bR^2$, hence $U$ is homeomorphic to a countable disjoint union of copies of $\bR^2$). There is a short exact sequence
\begin{center}
    % https://tikzcd.yichuanshen.de/?fbclid=IwAR0FybHEYl6QS_hMmNoKyCn60muRw7aO-eQAaFVrU3mcQkgAWpX_FaPCKwY#N4Igdg9gJgpgziAXAbVABwnAlgFyxMJZABgBpiBdUkANwEMAbAVxiRGJAF9T1Nd9CKAIzkqtRizYBhAPrEAFAFUAlFx4gM2PASIAmUdXrNWiEFPkAdCzgBOWOmADmDGKu68tAogGYD449LyANJu6pr8OigALH5GkqYcnGIwUI7wRKAAZjYQALZIZCA4EEhC7iDZeaXUxUi65ZX5iPpFJYjeDTlNvq1IUUmcQA
\begin{tikzcd}
0 \arrow[r] & C_0(U) \arrow[r] & C(\triangle) \arrow[r] & C(C) \arrow[r] & 0
\end{tikzcd}
\end{center}
Since $K^0(\bR^2)\cong K^0(\triangle)\cong\bZ$ and $K^{-1}(\bR^2)=K^{-1}(\triangle)=0$, the standard six-term exact sequence gives
\begin{center}
    % https://tikzcd.yichuanshen.de/?fbclid=IwAR0FybHEYl6QS_hMmNoKyCn60muRw7aO-eQAaFVrU3mcQkgAWpX_FaPCKwY#N4Igdg9gJgpgziAXAbVABwnAlgFyxMJZABgBpiBdUkANwEMAbAVxiRAB12AjALQD1OWMADMcATxABfUuky58hFACZyVWoxZtOvKTJAZseAkQAsq6vWatEIAMIAKANIBKXbMMLTpJWsuabxG76ckaKyCo+FhrWIIHS7vLGKGSR6lZsjgD6AIz2Di6ukmowUADm8ESgwgBOEAC2SGQgOBBI2fEgNfVt1C1ISh1dDYgArL2tiMSDtcNmzRMj092IAMzjSCZLwyrzSCtFkkA
\begin{tikzcd}
\bZ^\infty \arrow[rr] &  & \bZ \arrow[rr] &  & K^0(C) \arrow[dd] \\
                      &  &                &  &                 \\
K^{-1}(C) \arrow[uu]  &  & 0 \arrow[ll]   &  & 0 \arrow[ll]   
\end{tikzcd}
\end{center}
The top right arrow is surjective by exactness, and since $K^0(\triangle)\cong\bZ$ is generated by the class of $1_{C(\triangle)}$ and this is mapped to the class of the unit $1_{C(C)}$, it is also injective. Thus $K^0(C)\cong\bZ$, which implies that $K^{-1}(C)\cong\bZ^\infty$ by exactness.

Theorem \ref{ifs} therefore implies that $\alpha_*:K_0(C_r^*(R_{\phi_E}))\to K_0(C_r^*(R_{\tilde\phi}))$ is an isomorphism, and $K_1(C_r^*(R_{\tilde\phi}))\cong\bZ\oplus\bZ^\infty$ ($\bZ^\infty$ is free abelian, therefore the short exact sequence splits).
\end{example}

\begin{example}
Let $S$ denote the solid cube in $\bR^3$ with corners at $(i,j,k)$ for $i,j,k$ in $\{0,1\}$. For each $\delta$ in $\{0,1,2\}^3-\{(1,1,1)\}$, define the function $f_\delta:S\to S$ by $f_\delta(x)=\frac13(x+\delta)$. We let $C$ be the attractor of this iterated function system.

Let $(V,E)$ be the Bratteli diagram with $\#V_n=1$ for all $n$ and $\#E_n=26^n+1$ for all $n$. As before, any order on the edges will do, and assign the functions similarly. Before we apply Theorem \ref{ifs}, we verify that $K^0(C)\cong\bZ\oplus\bZ^\infty$ and $K^{-1}(C)=0$.

$C$ is obtained from $S$ by removing a countable number of open cubes, each of which is homeomorphic to $\bR^3$. Let $U$ denote the union of these open cubes. There is a short exact sequence
\begin{center}
    % https://tikzcd.yichuanshen.de/?fbclid=IwAR0FybHEYl6QS_hMmNoKyCn60muRw7aO-eQAaFVrU3mcQkgAWpX_FaPCKwY#N4Igdg9gJgpgziAXAbVABwnAlgFyxMJZABgBpiBdUkANwEMAbAVxiRGJAF9T1Nd9CKAIzkqtRizYBhAPrEAFAFUAlFx4gM2PASIAmUdXrNWiEFPkAdCzgBOWOmADmDGKu68tAogGYD449LyANJu6pr8OigALH5GkqYcnGIwUI7wRKAAZjYQALZIZCA4EEhC7iDZeaXUxUi65ZX5iPpFJYjeDTlNvq1IUUmcQA
\begin{tikzcd}
0 \arrow[r] & C_0(U) \arrow[r] & C(S) \arrow[r] & C(C) \arrow[r] & 0
\end{tikzcd}
\end{center}
Since $K^{-1}(\bR^3)\cong K^0(S)\cong\bZ$ and $K^0(\bR^3)=K^{-1}(S)=0$, we have
\begin{center}
    % https://tikzcd.yichuanshen.de/?fbclid=IwAR0FybHEYl6QS_hMmNoKyCn60muRw7aO-eQAaFVrU3mcQkgAWpX_FaPCKwY#N4Igdg9gJgpgziAXAbVABwnAlgFyxMJZABgBpiBdUkANwEMAbAVxiRAB12AjALQD1OWMADMcATxABfUuky58hFACZyVWoxZtOvKTJAZseAkQAsq6vWatEIAMIAKANIBKXbMMLTpJWsuabxG76ckaKyCo+FhrWIIHS7vLGKGSR6lZsjgD6AIz2Di6ukmowUADm8ESgwgBOEAC2SGQgOBBI2fEgNfVt1C1ISh1dDYgArL2tiMSDtcNmzRMj092IAMzjSCZLwyrzSCtFkkA
\begin{tikzcd}
0 \arrow[rr] &  & \bZ \arrow[rr] &  & K^0(C) \arrow[dd] \\
                      &  &                &  &                 \\
K^{-1}(C) \arrow[uu]  &  & 0 \arrow[ll]   &  & \bZ^\infty \arrow[ll]   
\end{tikzcd}
\end{center}
and therefore $K^0(C)\cong\bZ\oplus\bZ^\infty$ and $K^{-1}(C)=0$.

Theorem \ref{ifs} therefore implies that $K_0(C_r^*(R_{\tilde\phi}))\cong K_0(C_r^*(R_{\phi_E}))\oplus\bZ^\infty$, and that $\alpha_*:\bZ\to K_1(C_r^*(R_{\tilde\phi}))$ is an isomorphism.
\end{example}

\section*{Acknowledgements}
The content of this paper constitutes a portion of the research conducted for my PhD dissertation. I am grateful to my advisor Ian F. Putnam for suggesting the project to me, as well as for numerous helpful discussions. I also thank Aaron Tikuisis for spotting an error in the original proofs involving the $C$-extension property (Proposition \ref{cext} and Proposition \ref{cext2} of the current version). Finally, I thank the referee for their careful reading and suggestions.

\end{document}